\documentclass[runningheads]{llncs}
\usepackage[T1]{fontenc}
\usepackage{graphicx}
\usepackage[utf8]{inputenc} 
\usepackage{hyperref}
\usepackage{booktabs}
\usepackage{amsfonts}
\usepackage{amsmath}
\usepackage{amssymb}
\usepackage{epsfig}
\usepackage{mathtools}
\usepackage{mathrsfs}
\usepackage{enumitem}
\usepackage{subcaption}
\usepackage[linesnumbered,ruled,vlined]{algorithm2e}
\usepackage{color}

\newcommand{\N}{\mathbb{N}}  
\newcommand{\R}{\mathbb{R}}  
\renewcommand{\P}{\mathbb{P}}  
\newcommand{\D}{\mathbf{D}}
\newcommand{\E}{\mathbb{E}}  
\newcommand{\kernel}{\Pi}  

\newcommand{\chain}{\mathbf{X}}  
\newcommand{\initMeasure}{\nu}  
\newcommand{\invariantMeasure}{\mu}  
\newcommand{\roc}{\mathrm{ROC}}
\newcommand{\auc}{\mathrm{AUC}}

\newcommand{\githubRepo}{\url{https://github.com/statsmarkov/depth_markov}}

\NewDocumentCommand{\borel}{O{\Set}}{\mathcal{B}(#1)}  
\newcommand{\Set}{E}  
\newcommand{\SgB}{\borel[\Set]}  
\newcommand{\SetSeq}{\mathbb{T}}  
\newcommand{\SgE}{\mathcal{T}}  
\newcommand{\Depth}{D}  
\NewDocumentCommand{\DepthSeq}{O{\kernel}}{\Depth_{#1}}  

\newcommand{\Probabilities}{\mathcal{P}}  
\newcommand{\x}{\mathbf{x}}
\newcommand{\y}{\mathbf{y}}

\newcommand{\PropIndependenceInitDistrText}{(M0)}  
\newcommand{\PropIndependenceInitDistr}{\hyperref[prop:PropIndependenceInitDistr]{\PropIndependenceInitDistrText}}

\newcommand{\PropGeneralInvarianceText}{(P1)}  
\newcommand{\PropGeneralInvariance}{\hyperref[prop:PropGeneralInvariance]{\PropGeneralInvarianceText}}

\newcommand{\PropMarkovInvarianceText}{(M1)}  
\newcommand{\PropMarkovInvariance}{\hyperref[prop:PropMarkovInvariance]{\PropMarkovInvarianceText}}

\newcommand{\PropVanishingBoundaryText}{(P2)}  
\newcommand{\PropVanishingBoundary}{\hyperref[prop:PropVanishingBoundary]{\PropVanishingBoundaryText}}

\newcommand{\PropMarkovVanishingBoundaryText}{(M2)}  
\newcommand{\PropMarkovVanishingBoundary}{\hyperref[prop:PropMarkovVanishingBoundary]{\PropMarkovVanishingBoundaryText}}

\newcommand{\PropMaximalityCenterText}{(P3)}  
\newcommand{\PropMaximalityCenter}{\hyperref[prop:PropMaximalityCenter]{\PropMaximalityCenterText}}

\newcommand{\PropMonotonicityText}{(P4)}  
\newcommand{\PropMonotonicity}{\hyperref[prop:PropMonotonicity]{\PropMonotonicityText}}

\newcommand{\PropContinuityDepthText}{(P5)}  
\newcommand{\PropContinuityDepth}{\hyperref[prop:PropContinuityDepth]{\PropContinuityDepthText}}

\newcommand{\PropContinuityInProbUniformText}{(P6)}  
\newcommand{\PropContinuityInProbUniform}{\hyperref[prop:PropContinuityInProbUniform]{\PropContinuityInProbUniformText}}

\newcommand{\PropContinuityInProbText}{(P6')}  
\newcommand{\PropContinuityInProb}{\hyperref[prop:PropContinuityInProb]{\PropContinuityInProbText}}

\newcommand{\PropLipschitzText}{(P6'')}  
\newcommand{\PropLipschitz}{\hyperref[prop:PropLipschitz]{\PropLipschitzText}}

\newcommand{\PropWeakFellerText}{(P7)}  
\newcommand{\PropWeakFeller}{\hyperref[prop:PropWeakFeller]{\PropWeakFellerText}}

\begin{document}
\title{Anomaly Detection based on Markov Data: \\A Statistical Depth Approach}
%
%

\author{Carlos Fernández\inst{1}\orcidID{0000-0002-8577-865X} \and
    Stephan Clémençon\inst{1}\orcidID{0000-0002-5879-9500}}
\authorrunning{C. Fernández and S. Clémençon}
%
\institute{LTCI, Télécom Paris, Institut Polytechnique de Paris\\
    \email{\{fernandez,stephan.clemencon\}@telecom-paris.fr}}
\maketitle              
\begin{abstract}
  The purpose of this article is to extend
  the notion of statistical depth to the
  case of sample paths of a Markov chain. Initially introduced to
  define a center-outward ordering of points in the support of a multivariate
  distribution, depth functions permit to generalize the notions of quantiles and
  (signed) ranks for observations in $\mathbb{R}^d$ with $d>1$, as well as statistical
  procedures based on such quantities.
  Here we develop a general theoretical
  framework for evaluating the depth of a Markov sample path and recovering it
  statistically from an estimate of its transition probability with (non-) asymptotic
  guarantees. We also detail some of its applications, focusing particularly on unsupervised anomaly detection.
   Beyond the theoretical analysis carried
   out, numerical experiments are displayed, providing
   empirical evidence of the relevance of the novel concept we introduce here to quantify
   the degree of abnormality of Markov paths of variable length.
\end{abstract}
%
%
%

\section{Introduction}\label{sec:intro}
The concept of Markov chain provides a very popular time-series model, ubiquitous in
applications (\textit{e.g.} systems engineering, bioinformatics, mathematical finance,
operations research) to describe the dynamics governing the evolution of random time
phenomena in a parsimonious way, namely systems with a past, whose future distribution
depends on the most recent part only.

This paper is devoted to the extension of the \textit{statistical depth} concept,
originally introduced in \cite{Tukey1975} to overcome the lack of natural order on
$\mathbb{R}^d$ for $d\geq 2$, to sample paths with variable length of a Markov chain
$\mathbf{X}=(X_n)_{n\in \mathbb{N}}$. Whereas recent alternative approaches for
curve/functional data are essentially based on topological/metric properties of the path
space (see \textit{e.g.} \cite{StaermanMozharovskyiClemencon2020,GeenensNietoFrancisci2023})
and generally produce too large quantile regions when applied
to Markov data as will be shown later, the framework we develop is tailored to the Markov
setting. It is based on the description of the law of a Markov chain by its transition
probability kernel and can be considered as a very natural extension of the multivariate case.
Given a statistical depth on the state space, we propose to define the depth of a sample
path as the geometric average of the depth of each state involved in it w.r.t. the
transition probability evaluated at the preceding state. We show that the depth on the
torus of all finite length trajectories thus constructed inherits many of the properties
that the multivariate depth may enjoy, has significant computational advantages compared
to its competitors and can be interpreted in a very intuitive fashion: the greater the
average number of deep transitions, the deeper the path. The accuracy of sampling
versions of the Markov depth promoted is rigorously established in the form of (non-) asymptotic upper confidence bounds in the spirit of generalization guarantees in statistical learning theory.
While such results like ours have been mainly established in the i.i.d. case (see
\textit{e.g.} \cite{DLG96}), the case of observations with a dependency structure,
time-series in particular, has been the subject of intensive study in recent years, see
for instance \cite{AdamsNobel}, \cite{AlqWint2012}, \cite{ChristStein2009},
\cite{CCB20} \cite{Steinwart09}, \cite{KM14} or \cite{DiKolaczyk2004}. These results mainly
rely on extensions of probabilistic bounds for uniform deviations of i.i.d. averages from
their expectation to the case of \textit{weakly dependent training data}, under general
assumptions related to the decay rate of \textit{mixing coefficients}, see \cite{Rio17}.
However, few works, with the exception of \cite{CCB20} for instance, are specific to the
case of Markov chains and exploit their specific dynamics as we do in this paper.

Several promising applications related to the statistical analysis of Markov
paths of variable lengths are also discussed. Motivated by practical problems in various
domains (\textit{e.g.} health monitoring, quality control, security), we pay particular attention
to the task of detecting abnormal trajectories. In the numerical experiments that have been carried out,
the proposed methodology performs significantly better than pre-existing approaches,
paving the way for its use in a wider range of tasks.

The paper is structured as follows. In section \ref{sec:background}, the main notions of
statistical depth are briefly recalled alongside basic elements of Markov chain theory.
In section \ref{sec:markov_depth} we introduce the
concept of \textit{Markovian depth function},
we explain how to construct Markovian depth functions based on statistical depths
defined on the state space and we study its main properties.
Various numerical experiments illustrating the relevance of the concept we propose are
presented in section \ref{sec:numerical_experiments}. Section \ref{sec:perspectives}
collects a few concluding remarks and perspectives for further research, while, due to
space limitations, all technical proofs as well as additional discussions, examples and
applications are postponed to the Appendix.

\section{Background and Preliminaries}\label{sec:background}
\vspace{-5pt}As a first go, we briefly recall some basic notions pertaining to statistical depth
theory and key concepts in the study of Markov chains (refer to \textit{e.g.}
\cite{Meyn2009,markovChain2018} for a more detailed account), which shall be used in the
paper. Throughout the article, we denote by $\mathbb{I}\{B\}$ the indicator function of
any event $B$, by $\delta_a$ the Dirac mass at any point $a$ and by $\Rightarrow$
convergence in distribution. The Euclidean inner product and norm on $\mathbb{R}^d$ are
denoted by $\langle \cdot,\cdot \rangle$ and $\vert\vert \cdot\vert\vert$.

\subsection{Statistical Depth}\label{sec:statistical_depth}
Given the lack of any ``natural order'' on $\mathbb{R}^d$ when $d\geq2$, the notion of
\textit{statistical depth} allows to define a center-outward ordering of points in the
support of a probability distribution $P$ on $\mathbb{R}^d$ and, as a result, to extend
the notions of order and (signed) rank statistics to multivariate data, see
\cite{Mosler13}. By \textit{depth function} relative to $P$ is meant a function that
determines the centrality of any point $x\in \mathbb{R}^d$ with respect to the
probability measure $P$, that is to say any bounded non-negative Borel-measurable
function $D_P:\mathbb{R}^d\rightarrow \mathbb{R}_+$ such that the points $x\in
    \mathbb{R}^d$ near the 'center' of the mass are the deepest (\textit{i.e.} such that
$D_P(x)$ is among the highest values taken by the function).
The depth $D_P$ thus enables us to define a preorder for multivariate points $x \in \mathbb{R}^d$.
Originally introduced in the seminal contribution \cite{Tukey1975}, the
\textit{half-space depth} of $x$ in $\mathbb{R}^d$ relative to $P$ is the minimum of the
measure $P(\mathcal{H})$ taken over all closed half-spaces $\mathcal{H}\subset \mathbb{R}^d$
such that $x\in \mathcal{H}$. Many alternatives have been developed, see \textit{e.g.}
\cite{Liu1990}, \cite{LiuModarres2011}, \cite{mozharovskyi2022anomaly},
\cite{CuevasFraiman2009} or \cite{chernozhukov2017} among others. Refer to
\ref{subsec:depth_supp} in the Appendix for a list of popular statistical
depths for multivariate distributions and a discussion about computational issues. In
\cite{ZuoSerfling00}, an axiomatic nomenclature of statistical depths has been devised,
providing a systematic way of comparing their merits and drawbacks. Precisely, the four
properties below should be ideally satisfied for any distribution on
$\mathbb{R}^d$.

\begin{itemize}
    \item[\PropGeneralInvarianceText]\label{prop:PropGeneralInvariance} {\sc (Affine invariance)} Meaning by $P_X$ the law of any r.v. $X$ valued in $\mathbb{R}^d$, it holds: $D_{P_{AX+b}}(Ax+b)=D_P(x)$ for all $x\in \mathbb{R}^d$, any r.v. $X$ taking its values in $\mathbb{R}^d$, any $d\times d$ nonsingular matrix $A$ with real entries and any $b$ in $\mathbb{R}^d$.
    \item[\PropVanishingBoundaryText]\label{prop:PropVanishingBoundary} {\sc (Vanishing at infinity)} For any probability law $P$ on $\mathbb{R}^d$, the depth function $D_P$ vanishes at infinity,  \textit{i.e.} $D_P(x)\rightarrow 0$  as $\vert\vert x\vert\vert$ tends to infinity.
    \item[\PropMaximalityCenterText]\label{prop:PropMaximalityCenter} {\sc (Maximality at center)} For any probability distribution $P$ on $\mathbb{R}^d$ possessing a symmetry center $x_P$ (in a sense to be specified), the depth function $D_P$ takes its maximum value at it, \textit{i.e.} $D_P(x_P)=\sup_{x\in \mathbb{R}^d}D_P(x)$.
    \item[\PropMonotonicityText]\label{prop:PropMonotonicity} {\sc (Monotonicity relative to the deepest point)} For any distribution $P$ on $\mathbb{R}^d$ with deepest point $x_P$, the depth at any point $x$ in $\mathbb{R}^d$ decreases as one moves away from $x_P$ along any ray passing through it, \textit{i.e.}  $D_P(x)\leq D_P(x_P+\alpha(x-x_P))$ for any $\alpha$ in $[0,1]$.
\end{itemize}

Several studies have examined whether some of the properties listed above are verified by
the various notions of statistical depth proposed in the literature, see \textit{e.g.}
\cite{ZuoSerfling00}. The continuity
properties of the depths, in both \(x\) and \(P\), have also been studied for most
statistical depths, see \textit{e.g.} \cite{DonohoG92,NagyGijbels2016}.
As the distribution $P$ of interest is usually unknown in practice, its analysis must
rely on the observation of $N\geq 1$ independent realizations $X_1,\ldots, X_N$ of $P$. A
statistical version of $D_P(x)$ can be built by replacing $P$ with an empirical
counterpart $\hat{P}_N$ based on the $X_i$'s, \textit{e.g.} the raw empirical
distribution $(1/N)\sum_{i=1}^N\delta_{X_i}$, to get the \textit{empirical depth
    function} $\Depth_{\widehat{P}_N}(x)$. The consistency and asymptotic normality of
empirical depth functions have been analyzed for different notions of statistical depth \cite{DonohoG92},
and concentration results for
empirical half-space depth and contours can be found in \cite{BurrF17}.
The ranks and order statistics derived from a depth function can be used for a variety of
tasks, \textit{e.g.} classification and clustering \cite{jornsten}, anomaly
detection \cite{mozharovskyi2022anomaly} or two-sample tests \cite{ShiYueFu2023} among
others. The case of functional data has received some attention in the literature.
Although most practical approaches consist in projecting first the functional data onto
an appropriate finite dimensional space (\textit{filtering}) and using next a notion of
multivariate depth, certain methods recently documented rely on metrics or exploit the
geometry of the trajectories/curves in the path space,
see \cite{StaermanMozharovskyiClemencon2020}. As we will see in the following, such techniques perform poorly when applied to Markov data, given
the possibly great dispersion of Markov probability laws. In contrast to the methodology we
propose, none of them exploits the underlying Markov dynamics/structure of
paths and is suitable to analyze trajectories of different lengths.

\subsection{Basic Properties of Markov Chains}\label{sec:markov_chains}
We first recall a few definitions concerning the communication structure and the
stochastic stability of Markov chains. Let $\mathbf{X}=( X_{n}) _{n\in\mathbb{N}}$ be a
(time-homogeneous) chain (of order $1$) defined on some probability space
$(\Omega,\; \mathcal{F},\; \mathbb{P})$ with a countably generated state space
$E\subseteq \mathbb{R}^d$ equipped with its Borel $\sigma$-field $\mathcal{B}(E)$,
transition probability kernel $\kernel:(x, A)\in E\times \mathcal{B}(E) \mapsto \Pi_x(A)$
and initial probability distribution $\initMeasure$. For any $A\in \mathcal{B}(E)$ and
$n\in\mathbb{N},$ we thus have that \(X_{0}\sim\nu\) and
\begin{equation}\label{eq:Markov_ppty}
    \mathbb{P}(X_{n+1}\in A\mid X_{0},\ldots, X_{n})=\Pi_{X_{n}}(A)\quad a.s.
\end{equation}
We denote by $\mathbb{P}_{\nu}$ (respectively, by $\mathbb{P}_{x}$ for $x$ in
$E$) the probability measure on the underlying probability space
such that $X_{0}\sim\nu$ (resp. $X_{0}=x$) and by $\mathbb{E}_{\nu}[.]
$ the $\mathbb{P}_{\nu}$-expectation (resp. by $\mathbb{E}_{x}[.]$
the $\mathbb{P}_{x}$-expectation).

The chain $\mathbf{X}$ is said to be \textit{irreducible} if there exists a
$\sigma$-finite measure $\psi$ s.t. for all set $B\in \mathcal{B}(E)$, when $\psi(B)>0$,
the chain visits $B$ with strictly positive probability, no matter what the starting
point is. A Borel set $A$ is \textit{Harris recurrent} for the chain $\mathbf{X}$ if for any $x\in
    A$, $\mathbb{P}_{x}(\sum_{n=1}^{\infty}\mathbb{I}\{X_{n}\in A\}=\infty)=1$. The chain
$\mathbf{X}$ is said \textit{Harris recurrent }if it is $\psi $-irreducible and every
measurable set $A$ s.t. $\psi(A)>0$ is Harris recurrent. When the chain is Harris
recurrent, we have the property that
$\mathbb{P}_{x}(\sum_{n=1}^{\infty}\mathbb{I}\{X_{n}\in A\}=\infty)=1$ for any $x\in E$
and any $A\in \mathcal{B}(E)$ such that $\psi(A)>0.$ A probability measure
$\invariantMeasure$ on $E$ is said invariant for $\mathbf{X}$ when
$\invariantMeasure\Pi=\invariantMeasure$, where $\invariantMeasure\kernel(dy)=\int_{x\in
        E}\mu(dx)\kernel_x(dy) $. An irreducible chain is said \textit{positive recurrent} when
it admits an invariant probability (it is then unique).

\begin{example}\label{ex:modulated_random_walk}{\sc (Modulated random walk)}
    Consider the model defined by $X_{0}=x_0\geq 0$ and $X_{n+1}=\max(0,\; X_{n}+W_{n})$ for
    $n\in\mathbb{N}$ where $(W_{n})$ is a sequence of i.i.d. random variables with distribution \(F\).
    Then, $X_{n}$ is a Markov chain on $E=[0,\; \infty)$ with transition kernel: \(\kernel(x,[0,y]) = F(y-x)\quad x,y\geq 0.\)
    When \(\E [W_1] <0\), the chain is positive recurrent \cite[Proposition
        11.4.1]{Meyn2009}. Such a modulated random walk on the half-line models various
    systems, such as \textit{content-dependent storage processes}, \textit{work-modulated
        single server queues}, see \textit{e.g.} \cite[Section III.6]{Asmussen2010}.
\end{example}

Denote by $\mathbb{T}=\cup_{n\geq 1}E^n$ the torus of all trajectories of finite length
in the state space $E$, equipped with its usual topology $\mathcal{T}=\{O\subset
    \mathbb{T}:\; \forall n\geq 1,\; O\cap E^n\in \mathcal{B}(E^n)\}$ (see \cite[pp.
    131]{Dugundji1978} or Section \ref{subsec:topology_supp} of the Appendix
for further details). Any observed finite length path $(x_0,\ldots,x_n)$ of
$\chain$ is an element of $\SetSeq$. Our goal is to develop tools for assessing the
\textit{centrality}, or the \textit{outlyingness} conversely, of elements of \(\SetSeq\)
w.r.t. the law of a chain \(\chain\). As illustrated in the example
below, classic path depths based on the geometry of the space of trajectories are
inappropriate even in the case where all trajectories have the same length.

\begin{example}{\sc{(Limits of geometric approaches)}}\label{ex:limit_geometric}
    Consider the model described in Example \ref{ex:modulated_random_walk}, where \(W_n=-1.1\times
    Y_n+1\) and the \(Y_n\)'s are i.i.d. exponential r.v.'s with mean $1$. Note that,
    with probability $1$: \(W_n\leq 1\), hence \(X_{n+1}-X_n\leq 1\). Fig.
    \ref{fig:trajectories} shows, in solid red line, an unfeasible trajectory
    \(\textbf{\textnormal{x}}=(0,x_1,\ldots,x_5)\) where \(x_2-x_1=1.1\), sandwiched
    between two feasible trajectories (green/orange dashed lines). We have calculated the
    Lens and Mahalanobis depths (MBD) of this unfeasible trajectory w.r.t. the joint
    probability distribution of \(X_1,\ldots,X_5\), obtaining the values $0.42$ and
    $0.097$ respectively, very close to the depths of the feasible ones that surround it
    in Fig. \ref{fig:lens_depth} and \ref{fig:mahalanobis_depth}..
\end{example}
\vspace{-20pt}
\begin{figure}[ht]
    \centering
    \begin{subfigure}[b]{0.4\linewidth}
        \centering
        \includegraphics[height=3cm]{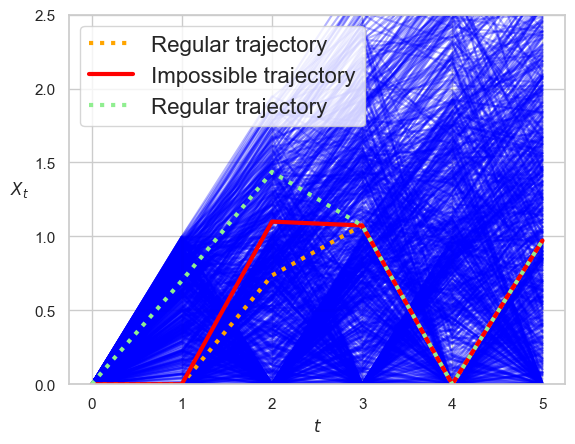}
        \caption{Trajectories}\label{fig:trajectories}
    \end{subfigure}
    \hfill
    \begin{subfigure}[b]{0.29\linewidth}
        \centering
        \includegraphics[height=3cm, width=\linewidth]{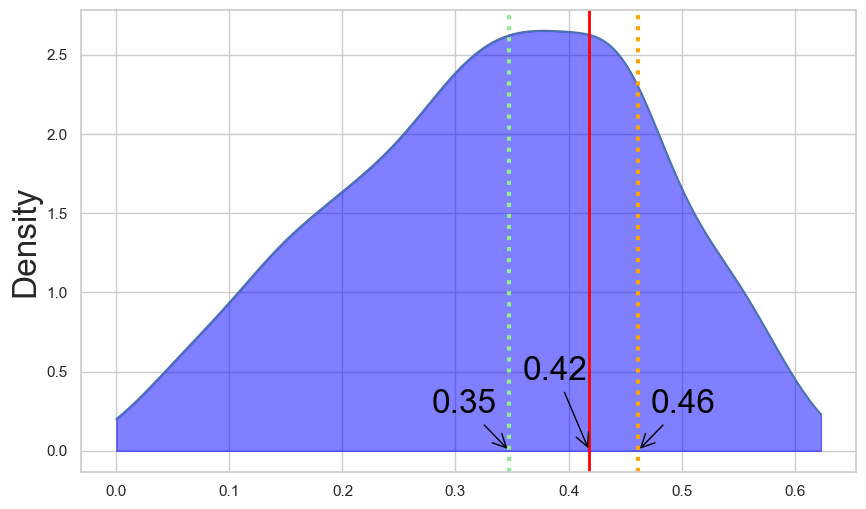}
        \caption{Lens depth density }\label{fig:lens_depth}
    \end{subfigure}
    \hfill
    \begin{subfigure}[b]{0.29\linewidth}
        \centering
        \includegraphics[height=3cm, width=\linewidth]{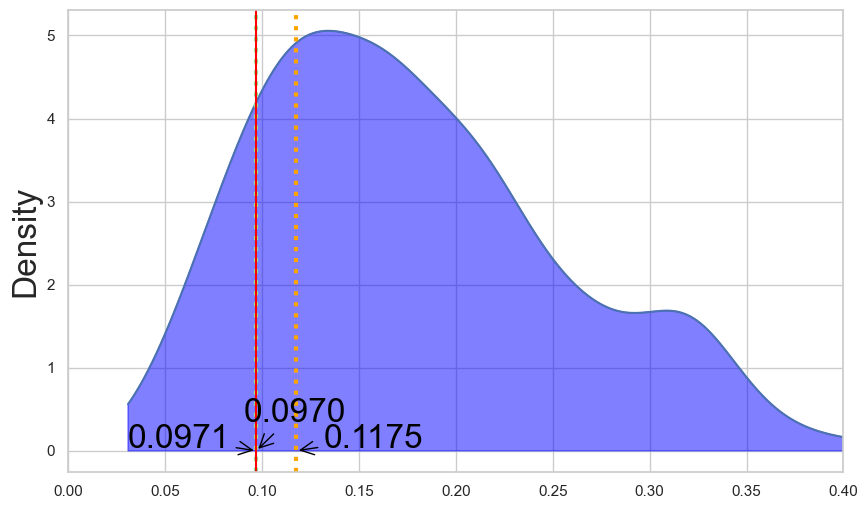}
        \caption{MBD density}\label{fig:mahalanobis_depth}
    \end{subfigure}

    \caption{(a) an unfeasible trajectory (red line) of the modulated random walk
        surrounded by two regular trajectories (green and orange
        dotted lines). (b) and (c), the densities of Lens and Mahalanobis depths w.r.t. the joint distribution of \(X_1,\ldots,X_5\).}\label{fig:metric_anomaly}
\end{figure}

\vspace{-20pt}\section{Depth functions for Markov Chains}\label{sec:markov_depth}
This section is devoted to extending the concept of statistical depth function
to the domain of Markov chains. We propose a very general mechanism for constructing such Markov depth functions and then investigate their main properties. Particular attention is paid to the study of sampling versions.

\subsection{Markovian Depth Functions}\label{sec:markov_depth_definition}
As stated in the introduction, our goal is to extend the notion of depth function to the
space \(\SetSeq\) of finite length trajectories of a Markov chain in order to assess the
centrality or outlyingness of Markov paths. The following is a list of a priori desirable
properties for such functions on $\mathbb{T}$.


\begin{definition}[Markovian depth function]\label{def:markov_depth} A statistical depth function w.r.t.
    the distribution of the Markov chain \(\chain\) is a function \(\D_{\chain}:\SetSeq\to \R_+\)
    that may satisfy some of the following properties:
    \begin{itemize}
        \item[\PropIndependenceInitDistrText]\label{prop:PropIndependenceInitDistr} {\sc (Independence from the initial law)}
              If \(\chain\) and \(\mathbf{X}'\) are two Markov chains with the same transition kernel, then:
              \begin{equation*}
                  \forall \mathbf{x}\in\SetSeq,\;\; \D_{\chain}(\mathbf{x})=\D_{\mathbf{X}'}(\mathbf{x}).
              \end{equation*}
        \item[\PropMarkovInvarianceText]\label{prop:PropMarkovInvariance} {\sc (Affine invariance)}
              For any $d\times d$ non singular matrix $A$, $b\in \mathbb{R}^d$, it holds:
              \begin{equation*}
\forall \mathbf{x}\in\SetSeq,\;\;                  \D_{A\chain+b}(A\mathbf{x}+b)=\D_{\chain}(\mathbf{x}),
              \end{equation*}
              where we set \(A\mathbf{x}+b=(Ax_0+b,\ldots,Ax_n+b)\) when \(\mathbf{x}=(x_0,\ldots,x_n)\).
        \item[\PropMarkovVanishingBoundaryText]\label{prop:PropMarkovVanishingBoundary} {\sc (Vanishing at infinity)} We have $\D_{\chain}(\mathbf{x})\to 0$ as $\mathbf{x}$ ``tends to infinity'' in
              the sense of $\SetSeq$'s topology $\SgE$.
    \end{itemize}

\end{definition}

Property {\PropIndependenceInitDistr} in the above definition guarantees that the depth
only depends on the dynamics, \textit{i.e.} the transition probability kernel $\kernel$,
which can be easily estimated with (non-) asymptotic guarantees based on the observation
of a single Markov path of sufficiently long length $N\geq 1$ (see \textit{e.g.}
\cite{Clem00} and \cite{Tjostheim-2001} and the references therein). Properties
    {\PropMarkovInvariance} and {\PropMarkovVanishingBoundary} are the natural extensions of
the properties {\PropGeneralInvariance} and {\PropVanishingBoundary} of the depth
functions on \(\Set\) to the path space $\mathbb{T}$. Regarding the extension of property
    {\PropMaximalityCenter}, though symmetry centers can be defined in the path space in
certain (seldom) situations, they do not necessarily correspond to realizable paths, as
illustrated in section \ref{sec:ex_ppties} of the Supplementary Material, in which case
the maximality at center property is absolutely not desirable. Notice also that, because
of the absence of any vector space structure on \(\SetSeq\), there is no natural
equivalent to the monotonicity property {\PropMonotonicity} in the Markovian case. In the next subsection we propose a simple way to build Markovian depth
functions on $\mathbb{T}$ based on statistical depth functions on \(\Set\).

\subsection{Markovian Sample Path Depth}\label{sec:markovian_depth}
Consider any notion of statistical depth $D$ on \(\Set\) and let
$\mathbf{x}=(x_0,\ldots,x_n)$ be a finite length realization of the chain $\mathbf{X}$
with initial distribution $\initMeasure$ and kernel \(\kernel\). For any
$i\in\{1,\ldots,n\}$, $x_i$ is a realization of the transition probability
$\Pi_{x_{i-1}}$, a distribution on $\Set$, and its depth can be naturally quantified
through $D_{\Pi_{x_{i-1}}}(x_i)$. We propose to define the depth of
$\mathbf{x}$ by aggregating the depth of all transitions forming the path
$\mathbf{x}$, as the geometric mean of the depths of each state transition.

\begin{definition}[Markovian path depth]\label{def:markovian_depth}
    With the convention that $\sqrt[n]{0}=0$ for $n\geq 1$, the Markov depth function of
    a chain $\mathbf{X}$ with transition $\Pi$ based on the multivariate statistical
    depth $D$ is the real-valued function defined on the torus $\mathbb{T}$ of all
    trajectories of finite length as: $\forall n\geq 1$, $\forall \mathbf{x}=(x_0,\;
        \ldots,\; x_n)\in E^{n+1}$,
    \begin{equation}\label{eq:markovian_depth}
        D_{\Pi}(\mathbf{x})=\sqrt[n]{\prod_{i=1}^{n}D_{\Pi_{x_{i-1}}}(x_i)}.
    \end{equation}
\end{definition}

This definition extends that of statistical depth for multivariate distributions. Indeed,
when $n=1$, we have $D_{\Pi}((x_0,x_1))=D_{\Pi_{x_{0}}}(x_1)$. The rationale behind it is
obvious: \textit{the deeper in average the transitions forming the path, the deeper the
    trajectory}. In particular, if an impossible transition $x_{i-1}\to x_i$ is present in
the trajectory $\mathbf{x}$ (\textit{i.e.} $D_{\Pi_{x_{i-1}}}(x_i)=0$), its Markov depth
equals $0$, \textit{i.e.} $D_{\Pi}(\mathbf{x})=0$. Provided that the depth functions
$D_{\Pi_x}$ can be easily computed (or estimated, see the analysis below), the Markovian
sample path depth permits to evaluate and compare the depth of paths of any (finite)
length.

\begin{example}{\sc (Markovian half-space depth)} If the Markov depth is based on the half-space
    depth introduced in \cite{Tukey1975} (see also section \ref{subsec:depth_supp} in the Appendix),
    it takes the form: $\forall n\geq 1$, $\forall \mathbf{x}=(x_0,\; \ldots,\; x_n)\in E^{n+1}$,
    \begin{equation}
        D_{\Pi}(\mathbf{x})=\sqrt[n]{\prod_{i=1}^n \inf_{u\in \mathbb{S}_{d-1}}\Pi_{x_{i-1}}\left(\mathcal{H}_{u,{x_i}}\right)},
    \end{equation}
    where $\mathcal{H}_{u,x}=\{y\in \mathbb{R}^d:\; \langle u,\; y-x\rangle \geq 0\}$ for
    all $(u,x)\in \mathbb{S}_{d-1}\times E$.
\end{example}

\begin{example}{\sc (Markov IRW-depth)} In order to avoid optimization issues, various alternatives to Tukey's
    depth functions have been proposed, see \textit{e.g.} \cite{CuevasFraiman2009}.
    In \cite{IRW} (see also \cite{10.1214/23-EJS2189}), a new data depth, referred to as the Integrated
    Rank-Weighted (IRW) depth, is defined by using an integral over the sphere
    $\mathbb{S}^{d-1}$ instead of the infimum. When based on the IRW-depth, the Markov depth is
    given by: $\forall n\geq 1$, $\forall \mathbf{x}=(x_0,\; \ldots,\; x_n)\in E^{n+1}$,
    \begin{equation}\label{eq:IRW}
        D_{\Pi}(\mathbf{x})=\sqrt[n]{\prod_{i=1}^n \int_{u\in \mathbb{S}_{d-1}}\Pi_{x_{i-1}}\left(\mathcal{H}_{u,{x_i}}\right)\omega_{d-1}(du)},
    \end{equation}
    where $\omega_{d-1}$ is the spherical probability measure on $\mathbb{S}_{d-1}$. The
    main advantage of the IRW depth is of computational nature, the integrals w.r.t.
    $\omega_{d-1}$ involved in \eqref{eq:IRW} can be approximated by Monte Carlo
    procedures: $Z/\vert\vert Z\vert\vert \sim \omega_{d-1}$ as soon as
    $Z\sim\mathcal{N}(0,\; \mathcal{I}_d)$ is a standard centered Gaussian.
\end{example}

In addition to its interpretability, the concept introduced has a number of advantages,
as it inherits many properties from the notion of multivariate statistical depth it
relies on (see Theorems \ref{thm:ZS_properties}, \ref{th:continuity_in_x} and
\ref{th:non_asymptotic_bound}). As discussed in the following, the choice of $D$ must be
weighed against the verification of the desirable properties and the availability of
efficient algorithms for computation in the case of empirical distributions
(\textit{i.e.} for computing $D_{\hat{P}}$, where $\hat{P}$ is a statistical version of
the distribution $P$ considered).\smallskip

\noindent {\bf Main Properties}.
Before investigating the theoretical properties of the Markovian path
depth and inference issues, we state the following result, revealing that the Markov
depth of a positive recurrent chain $\mathbf{X}$, evaluated at a random path of length
$n+1$, is \textit{stochastically stable} as $n\to \infty$, meaning that its distribution
asymptotically converges to a Gaussian limit as the path length $n$ increases.

\begin{proposition}\label{prop:limit}{\sc (Limit distribution)}
  Assume that $\mathbf{X}$ is positive recurrent with stationary distribution
  $\invariantMeasure$ and that $\log(D_{\Pi_x}(y))$ is integrable w.r.t.
  $\invariantMeasure d(x)\Pi_x(dy)$ in $E^2$. Then, for any initial distribution $\nu$,
  we have that $D_{\Pi}(X_0,\; X_1,\; \ldots,\; X_n)$ converges almost surely to $D_{\infty}(\Pi)$
  as $n\rightarrow \infty$, where
  \begin{equation}
    D_{\infty}(\Pi)= \exp\left(\int_{(x,y)\in E^2} \log\left(D_{\Pi_x}(y)  \right)\Pi_{x}(dy)\invariantMeasure(dx) \right).
  \end{equation}
  If in addition $\log(D_{\Pi_x}(y))$ is square integrable w.r.t.
  $\invariantMeasure d(x)\Pi_x(dy)$, then we have that, as $n$ goes to $\infty$,
  \begin{equation*}
    \sqrt{n}\Bigl(D_{\Pi}(X_0,X_1,\ldots,X_n)- D_{\infty}(\Pi)\Bigr)
  \end{equation*}
  converges in distribution to a centered normal random variable. The exact form of the asymptotic variance
  is given by \eqref{eq:asymptotic_vaiance} in the Supplementary Material.
\end{proposition}

Our next result shows that the Markovian sample path depth \(\DepthSeq\) inherits most of
its statistical properties from the multivariate depth $D$ it is based on, and for most
choices of \(\Depth\), it satisfies the properties given in Definition
\ref{def:markov_depth}.

\begin{theorem}\label{thm:ZS_properties} Let \(\Depth\) be a statistical depth on \(\Set\) and
  \(\DepthSeq\) be the Markovian sample path depth based on the latter. The following assertions hold true.
  \begin{itemize}
    \item[(i)] The Markovian depth \(\DepthSeq\) satisfies {\PropIndependenceInitDistr}.
    \item[(ii)] If \(\Depth\) satisfies {\PropGeneralInvariance}, then \(\DepthSeq\) satisfies {\PropMarkovInvariance}.
    \item[(iii)] Assume that $D$ fulfills the following stronger
          version of property {\PropVanishingBoundary}: $\forall M>0$,
          $\sup_{||x||\leq M}D_{\Pi_x}(y)\to 0$ as $\vert\vert y\vert\vert \to \infty$.
          Then \(\DepthSeq\) satisfies {\PropMarkovVanishingBoundary}.
  \end{itemize}
\end{theorem}

In addition to the above properties, \(\DepthSeq\) also satisfies the following property,
which can be viewed as a variation of {\PropMaximalityCenter}.

\begin{proposition}\label{thm:MaximalityCenter}{\sc (Maximality at paths of centers)} Let $\Pi$ be a transition probability on $\mathbb{R}^d$. Assume that $D$
  fulfills property {\PropMaximalityCenter} for a given notion of symmetry, that
  the distribution $\Pi_x(dy)$ on $\Set$ is 'symmetric' w.r.t. a center
  $\theta(x)\in \Set$ for any $x\in \Set$ and that $x\in
    \Set\mapsto D_{\kernel_{x}}(\theta(x))$ is constant. Then, for any initial
  value $x_0\in \mathbb{R}^d$ and any $n\geq 1$, the trajectory $(x_0,\;
    \theta^{(1)}(x_0),\; \ldots,\; \theta^{(n-1)}(x_0))$ is of maximal Markov depth
  $D_{\Pi}$, where $\theta^{(1)}(x_0)=\theta(x_0)$ and $\theta^{(i+1)}(x_0)=(\theta\circ \theta^{(i)})(x_0)$ for $i\geq 1$.
\end{proposition}

We point out that the paths considered in Proposition \ref{thm:MaximalityCenter} above
are not symmetry centers for $\mathbf{X}$'s law in general.

\medskip

\noindent {\bf Continuity (in \(\x\), in $\Pi$).} The following classic technical conditions are involved in the subsequent study of the continuity properties of the Markov depth.
\begin{itemize}
  \item[\PropContinuityDepthText]\label{prop:PropContinuityDepth} {\sc (Continuity in \(x\))} For any law $P$ on $\mathbb{R}^d$, the mapping $x\in \mathbb{R}^d\mapsto D_P(x)$ is continuous.
  \item[\PropContinuityInProbUniformText]\label{prop:PropContinuityInProbUniform} {\sc (Uniform weak continuity in $P$)} For any probability distributions $P$ and $P_n$ on $\mathbb{R}^d$ s.t. $P_n\Rightarrow P$ as $n\to \infty$, we have $              \sup_{x\in\mathbb{R}^d}|\Depth_{P_n}(x)-\Depth_P(x)|\to 0$ as $n\to \infty$.
  \item[\PropContinuityInProbText]\label{prop:PropContinuityInProb} {\sc (Weak continuity in \(P\))}
        For any probability distributions $P$ and $P_n$ on $\mathbb{R}^d$ s.t. $P_n\Rightarrow P$ as $n\to \infty$, we have: $\forall x\in\mathbb{R}^d$, $|\Depth_{P_n}(x)-\Depth_P(x)|\to 0$ as $n\to \infty$.
  \item[\PropWeakFellerText]\label{prop:PropWeakFeller} {\sc (Weak Feller condition)} Denoting by $\mathcal{M}_1(E)$ the set of probability measures on $E$ equipped with the weak convergence topology, the function $x\in E \mapsto \Pi_x$ mapping $E$ to $\mathcal{M}_1(E)$ is continuous.
\end{itemize}

Properties {\PropContinuityDepth} and {\PropContinuityInProbUniform} are satisfied by
many notions of statistical depth (see section \ref{subsec:depth_supp}). The result below
reveals that the Markov depth inherits its continuity properties from those of the
multivariate depth $D$ involved in its definition.

\begin{theorem}\label{th:continuity_in_x}
  The following assertions hold true.
  \begin{itemize}
    \item[(i)]  {\sc (Continuity in $\mathbf{x}$)} Suppose that properties {\PropContinuityDepth}, {\PropContinuityInProbUniform} and {\PropWeakFeller} are
          fulfilled. Then, the mapping $\mathbf{x}\in \mathbb{T}\mapsto D_{\Pi}
            (\mathbf{x})$ is continuous.

    \item[(ii)] {\sc (Continuity in $\Pi$)} Suppose that {\PropContinuityInProb} holds. Let $\Pi$ and $\Pi^{(n)}$ be transition probabilities on $E$ s.t. $\forall x\in E$, $\Pi^{(n)}_x\Rightarrow \Pi_x$ as $n\to \infty$. Then: $\forall \mathbf{x}\in \mathbb{T}$, $D_{\Pi^{(n)}}(\mathbf{x})\to D_{\Pi}(\mathbf{x})$.
  \end{itemize}
\end{theorem}

\noindent{\bf Sampling Versions and Estimation.}
The Markovian depth $D_{\Pi}$ is unknown in general, just like the transition probability
$\Pi$. In practice, a \textit{plug-in} strategy based on an empirical counterpart
$\hat{\Pi}$ of $\Pi$ and must be implemented, using $D_{\hat{\Pi}}$ as an estimate of
$D_{\Pi}$. Refer to \ref{subsec:kernel_supp} in the Supplementary Material for a
description of inference techniques dedicated to the statistical estimation of
\(\kernel\) with (non-asymptotic) guarantees. The following condition, which can be
viewed as strong version of {\PropContinuityInProbUniform}, permits to link $\kernel$'s
estimation error with that of $D_{\kernel}$.
\begin{itemize}
  \item[\PropLipschitzText]\label{prop:PropLipschitz}{\sc (Lipschitz condition)} Let $\mathcal{A}$ be a collection of Borel subsets of $\mathbb{R}^d$. There
        exists a finite constant $C_d$ such that, for any probability distributions $P$
        and $Q$ on $\mathbb{R}^d$, we have: $ \sup_{x\in \mathbb{R}^d}|\Depth_P(x) -
          \Depth_Q(x)| \leq C_d ||P-Q||_{\mathcal{A}}$, where
        $||P-Q||_{\mathcal{A}}:=\sup_{A\in \mathcal{A}}|P(A)-Q(A)|$.
\end{itemize}

Property {\PropLipschitz} is fulfilled by various multivariate depths for specific
classes $\mathcal{A}$ (see Section \ref{subsec:depth_supp} of the appendix). As shown by the theorem below,
non-asymptotic guarantees for sampling versions of the Markov depth can be established
under this condition.

\begin{theorem}\label{th:non_asymptotic_bound}
  Let $n\geq 1$ and $\mathbf{x}\in E^{n+1}$. Consider two
  transition probabilities $\Pi$ and $\hat{\Pi}$ on $E$. Suppose that \(\Depth\) fulfills {\PropLipschitz} and there exists \(\epsilon>0\) s.t. $\min
    \{D_{\hat{\kernel}_{i}}(x_{i+1}),D_{\kernel_{i}}(x_{i+1})\}>\epsilon$ for $i<n$. We have
  \begin{equation}\label{eq:bound}
    |\DepthSeq(\x) - \DepthSeq[\hat{\kernel}](\x)| \leq \frac{7}{4}C_d\frac{\DepthSeq(\x)}{\epsilon} \max_{i=0,\; \ldots,\; n-1}||\kernel_{x_i} - \hat{\kernel}_{x_i}||_{\mathcal{A}}.
  \end{equation}
\end{theorem}

The above result deserves comments. It proves that error bounds (in expectation, in
probability) for point-wise deviations between $D_{\Pi}(\mathbf{x})$ and
$D_{\hat{\Pi}}(\mathbf{x})$ can be straightforwardly deduced from bounds for the maximal
deviations over the class $\mathcal{A}$ between the transition probability $\Pi_x$ and
its estimator $\hat{\Pi}_x$ at the successive points $x$ forming the path $\mathbf{x}$
considered. In addition, as $\min_{1\leq i<n}D_{\Pi_{x_{i-1}}}(x_i)$ gets smaller, the factor
$D_{\Pi}(\mathbf{x})/\epsilon$ on the right hand side of \eqref{eq:bound} increases to
$\infty$, showing that the smaller the Markov depth value to be estimated with a given
error bound, the higher the accuracy of the transition probability estimator required.\smallskip

\noindent{\bf Computational issues.}
Various efficient algorithmic procedures have been designed to compute the values taken
by a depth function in the multivariate case, based on a discrete, empirical distribution \cite{MM22}.
The approach below shows how to use the latter to estimate the values taken by
a Markovian sample path depth.
\vspace{-25pt}
\begin{algorithm}
  \DontPrintSemicolon
  \KwIn{Path $\x=(x_0,\ldots,x_n)$ with $n\geq 1$ in $\mathbb{T}$, transition probability $\hat{\Pi}$ and precision control integer \(M\geq 1\).}

  \For{$i:= 0$ \KwTo $n-1$}{

  Generate $M$ independent samples $x_{1,i},\ldots,x_{M,i}$ from $\hat{\kernel}_{x_i}$\;
  Compute an estimator $\widehat{D}_i$ of $\Depth_{\hat{\kernel}_{x_i}}(x_{i+1})$ based
  on \(y_{1,i},\ldots,y_{M,i}\)\;

  }

  \KwOut{$\Depth_{\hat{\kernel}}(\x)=\sqrt[n]{\prod_{i=1}^{n}{\widehat{D}_i}}$.}

  \caption{Estimation of $D_{\hat{\Pi}}(\x)$}
  \label{alg:estimating_markov_depth}
\end{algorithm}

\vspace{-15pt}Attention should be paid to the fact that the complexity of estimating the value that
$D_{\hat{\Pi}}$ takes at a path of length $n+1$ through Algorithm
\ref{alg:estimating_markov_depth} is of order $O(ng(M,d))$, where $O(g(M,d))$ is the
numerical complexity of computing the value taken by the empirical depth function
$\Depth_{\hat{P}}$ based on \(M\) independent samples of \(P\). For instance, we have
$g(M,d)=M^{d-1}\log M$ in the case of the half-space depth, see \textit{e.g.}
\cite{DyckerhoffM16}. In comparison, the complexity of calculating the value
taken by the multivariate depth at $(x_0,\ldots,x_n)$, viewed as a point of
$\mathbb{R}^{(n+1)d}$, based on $M$ paths of length $n+1$ (viewed as $M$ sample points in
$\mathbb{R}^{(n+1)d}$) is of order $O(g(M,(n+1)d))$ (namely, $M^{(n+1)d-1}\log M$ in the
case of the half-space depth). As a result, the Markov depth offers considerable
computational advantages over the direct use of multivariate depth when considering
``long'' trajectories.



\section{Application to Anomaly Detection}\label{sec:numerical_experiments}
The proposed notion of Markov depth can be used for various tasks related to
the statistical analysis of Markov paths of different lengths. For reasons of space and
given the importance of this task in practice, the focus is on (unsupervised)
anomaly detection in this section. Applications of the Markov depth to the clustering of Markov paths of
different lengths and to the problem of testing whether two collections of Markov
trajectories are drawn from the same (unknown) transition probability are described in
the appendix, together with dedicated numerical experiments.
\medskip

\noindent \textbf{Unsupervised anomaly detection.} Markov chains are used in a wide variety of fields, including systems engineering to
model storage systems or tele-traffic data as well as time-series analysis,
and (unsupervised) anomaly detection plays a crucial role in the monitoring/management of
such complex systems/phenomena. Precisely, the problem considered is as follows. The
objective is to identify paths $\mathbf{x}$ that are suspicious due to their significant
difference with the majority of observed sample paths. Here, an observation
$\mathbf{x}=(x_0,\ldots,x_n)$ with fixed length $n\geq 1$ and initial value $x_0\in
    \mathbb{R}^d$ is \textit{abnormal} when it is not the realization of the Markov chain
with transition probability $\Pi$ generating the \textit{normal} (``not abnormal'', with no
reference to the Gaussian law here) trajectories, \textit{i.e.} when it is not a
realization of the probability distribution
$\delta_{x_0}(dx_{0})\Pi(x_0,dy_1)\Pi(y_1,dy_2)\ldots \Pi(y_{n-1},dy_n)$, $\Pi$ being
unknown in practice of course.
Based on a set of (unlabeled, but supposedly normal in vast majority) training paths, the
goal is to build an anomaly scoring rule $s:\mathbb{T}\to \mathbb{R}$ permitting to
assign a level of abnormality $s(\mathbf{x})$ to any future path $\mathbf{x}$ in
$\mathbb{R}$: the lower the score $s(\mathbf{x})$, the more suspicious the path
$\mathbf{x}$ observed ideally. Although the learning stage involves unlabeled datasets,
the availability of labels is required to assess predictive performance (possibly
depending on the type of anomaly considered). The gold standard in this respect is the
$\roc$ curve, \textit{i.e.} the parameterized curve describing the rate of normal paths
with a score below $t$ (false positive rate) compared to that of abnormal paths with a
score below $t$ (true positive rate), as the $t$ threshold varies, or its popular scalar
summary, the $\auc$ criterion (\textit{i.e.} the Area Under the $\roc$ Curve).

Below, we describe the two experimental contexts for which we then present the numerical
results. The latter reveal that the performance of the simple method of using an
empirical version of a Markov depth as the anomaly scoring function is quite similar to
that of conventional techniques (which can only be used in contrast when the
training/test paths are all of the same length) when considering well-identified
isolated/displacement/shape anomalies (see \cite{hubert2015multivariate}) and
significantly outperforms them when the anomaly is of a dynamic nature (\textit{i.e.}
corresponding to a more or less long-lasting change in transition dynamics), in line with
its original motivation. This is confirmed by the extra experiments documented in
the Appendix.

\noindent\textbf{Nonlinear time-series.} Consider the ARCH(1) Markov chain, defined
by: \(X_{n+1}=m(X_{n})+\sigma(X_n)\epsilon_n\) where \(m:\R\mapsto\R\) and
\(\sigma:\R\mapsto\R^*_{+}\) are unknown measurable functions and \(\epsilon_n\) is a
sequence of i.i.d. r.v.'s with mean \(0\) and variance \(1\) independent of
\(X_0\).
Here, we have chosen \(X_0\equiv 0.5\), \(m(x)={1/(1+\exp(-x))}\) and
\(\sigma(x)=\psi(x+1.2)+1.5\psi(x-1.2)\) where \(\psi(x)\) is $\mathcal{N}(0,1)$'s
density function and $\epsilon_n \sim \mathcal{N}(0,1)$. This model, introduced
in \cite{Engle1982}, is widely used in Econometrics, to model the return of
financial assets \cite{BollerslevChouKroner1992,HARDLE1997223,FanYao1998}.\smallskip

\noindent {\bf Queuing system.} Consider now a GI/G/1 queuing system with
interarrival times \(T_n\) and service times \(V_n\) where \(\{T_n\}_{n\geq 0}\) and
\(\{V_n\}_{n\geq 0}\) are independent i.i.d. sequences with distributions \(V\) and \(T\)
respectively. Denote by \(X_n\) the waiting time of
the \(n\)-th customer and assume that \(X_0=0\), then, \(X_{n+1}=\max(0,X_{n}+W_{n})\) where \(W_n=V_n-T_n\), which
corresponds to
Example \ref{ex:modulated_random_walk}. Here, \(V\) and \(T\) are
exponential r.v.'s with means \(0.45\) and, \(0.5\) respectively.
\smallskip

\noindent\textbf{Generation of anomalous paths.} We examine $4$ types of anomalies (isolated/shock, two dynamic
and shift), generated by modifying the Markov model over a
path segment. For the ARCH(1) model, the isolated anomaly is simulated by setting
\(m(x)=5x\) and \(\sigma(x)=|x|^{1/2}\) for two steps. The first dynamic anomaly alters
\(m(x)=1/{(2+\exp(-x))}\) for 60\% of the path, while the second takes
\(\sigma(x)=0.5\sqrt{x^2+1}\) over 60\% of the trajectory. The shift anomaly uses
\(m(x)\equiv 2\) for 60\% of the path length.

For the queuing system, the shock anomaly is generated by using \(V\) as an exponential
distribution with a mean of $2.25$ over 10\% of the trajectory. The first dynamic anomaly
changes the interarrival time distribution \(T\) to an exponential with a mean of $0.1$
over 20\% of the path, simulating a period of increased customer arrivals. The second
dynamic anomaly modifies the service time distribution \(V\) to \(0.55 \mathcal{U}\),
where \(\mathcal{U}\) is uniform on $(0,2)$, over 30\% of the path, representing a period
of slightly slower service times. The shift anomaly is generated by considering
deterministic interarrival times \(T_n=2^{-n}\) over 25 of steps.

\vspace{-26pt}
\begin{table}[ht]
    \caption{$\auc$ for anomalies in the ARCH(1) and queuing models}\label{tab:combined_auc}
    \vspace{5pt}
    \centering
    \begin{tabular}{lcc}
        \toprule
        Anomaly type       & ARCH(1) & Queuing \\
        \midrule
        Shock              & 0.97    & 0.95    \\
        Dynamic anomaly I  & 0.71    & 0.90    \\
        Dynamic anomaly II & 0.87    & 0.73    \\
        Shift              & 1.00    & 0.93    \\
        \bottomrule
    \end{tabular}
\end{table}

\vspace{-10pt}For each model we simulated one long training path (\(n=1000\)) and $4$ contaminated data
sets, each one containing \(200\) paths of random length (between \(50\) and \(200\))
where 50\% of those contain a specific type of anomaly. Based on the long path, a kernel
estimator \(\hat{\kernel}\) (see \ref{subsec:kernel_supp}) is computed and used to infer,
for each trajectory, the Markov depth \(\Depth_{\kernel}\) based on the half-space depth.

\vspace{-10pt}\begin{figure}[ht]
    \centering
    \begin{subfigure}[b]{0.24\linewidth}
        \centering
        \includegraphics[width=\linewidth]{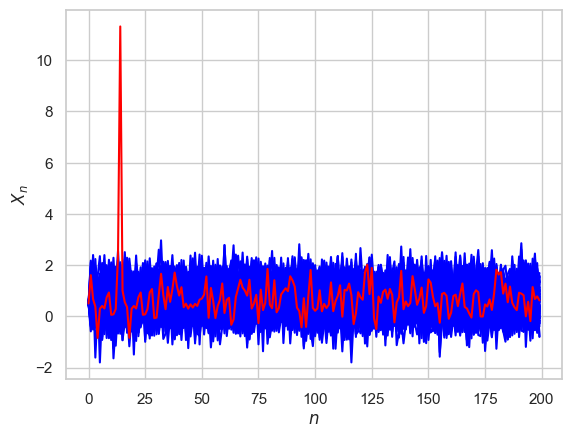}
        \caption{Shock}\label{fig:arch_shock}
    \end{subfigure}
    \hfill
    \begin{subfigure}[b]{0.24\linewidth}
        \centering
        \includegraphics[width=\linewidth]{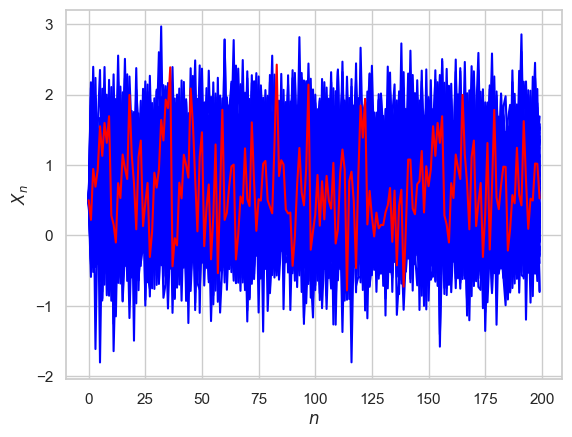}
        \caption{Dyn. an. II}\label{fig:arch_dyn_an_2}
    \end{subfigure}
    \hfill
    \begin{subfigure}[b]{0.24\linewidth}
        \centering
        \includegraphics[width=\linewidth]{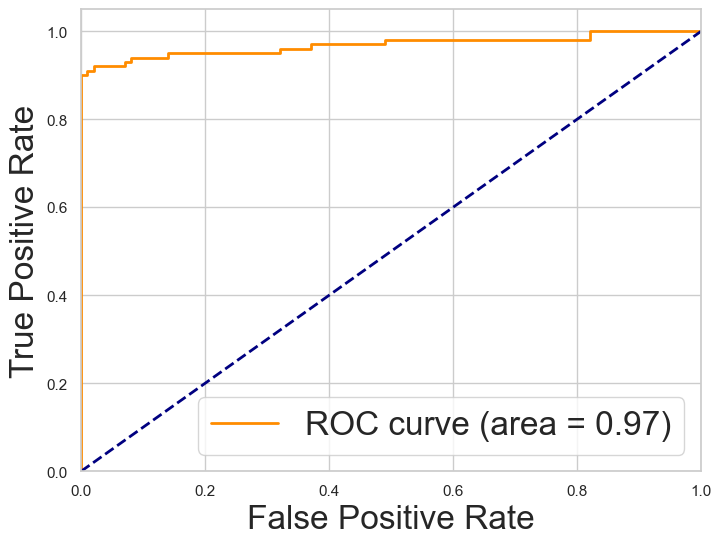}
        \caption{\(\roc\) (Shock)}
    \end{subfigure}
    \hfill
    \begin{subfigure}[b]{0.24\linewidth}
        \centering
        \includegraphics[width=\linewidth]{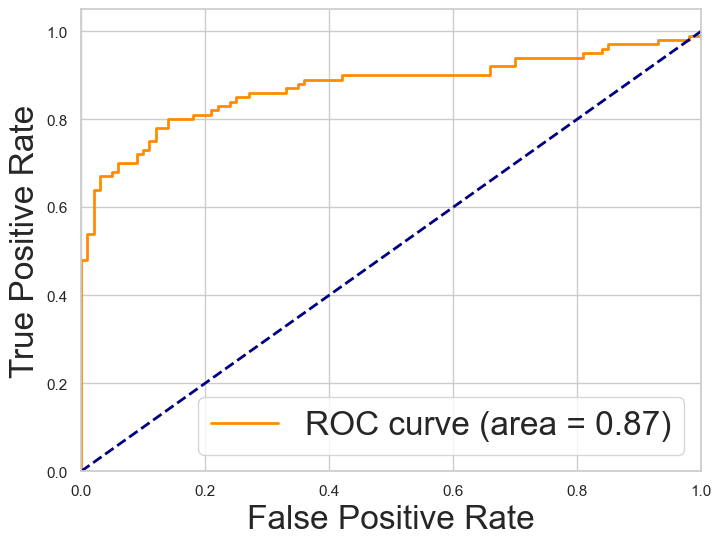}
        \caption{\(\roc\) (Dy. an. II)}
    \end{subfigure}
    \caption{Subfigures (a) and (b) illustrate anomalies in the ARCH(1) model,
        with normal paths shown in blue and anomalous trajectories in red.
        Subfigures (c) and (d) present the $\roc$ curves of our scoring function for the shock and
        dynamic anomaly II scenarios.}\label{fig:arch_compendium}
\end{figure}

\vspace{-15pt}\noindent \textbf{Results and Discussion.} The results in Table 1 indicate that the Markov depth detects with
high accuracy classic anomalies such as shocks/shifts, see Fig. \ref{fig:arch_shock} and \ref{fig:qeueue_shift}. Our approach is
also capable of detecting the more challenging dynamic anomalies, see Fig.
\ref{fig:arch_dyn_an_2} and \ref{fig:qeueue_dyn_an_1}. The corresponding $\roc$ curves in
Fig. \ref{fig:arch_compendium} and \ref{fig:queue_compendium} reveal its strong
performance on the different types of anomaly considered. In Section
\ref{sec:comparison_competitors}, it is compared with that of alternative anomaly
detection techniques, focusing on fixed-length paths since the competitors cannot
straightforwardly handle variable-length paths. The results of these experiments, presented in
Table \ref{tab:both_auc_comparison} and Fig.
\ref{fig:arch_roc_comparison} and \ref{fig:queue_roc_comparison} in the Appendix, show that our method performs similarly to the best of the competitors for
classic anomalies and significantly outperforms them for dynamic ones.

\begin{figure}[ht]
    \centering
    \begin{subfigure}[b]{0.24\linewidth}
        \centering
        \includegraphics[width=\linewidth]{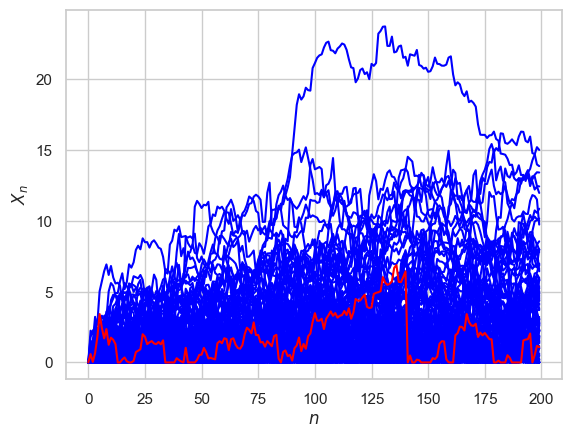}
        \caption{Dy. an. I}\label{fig:qeueue_dyn_an_1}
    \end{subfigure}
    \hfill
    \begin{subfigure}[b]{0.24\linewidth}
        \centering
        \includegraphics[width=\linewidth]{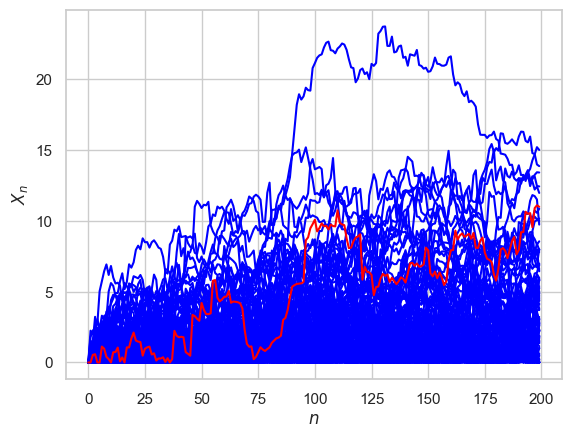}
        \caption{Shift}\label{fig:qeueue_shift}
    \end{subfigure}
    \hfill
    \begin{subfigure}[b]{0.24\linewidth}
        \centering
        \includegraphics[width=\linewidth]{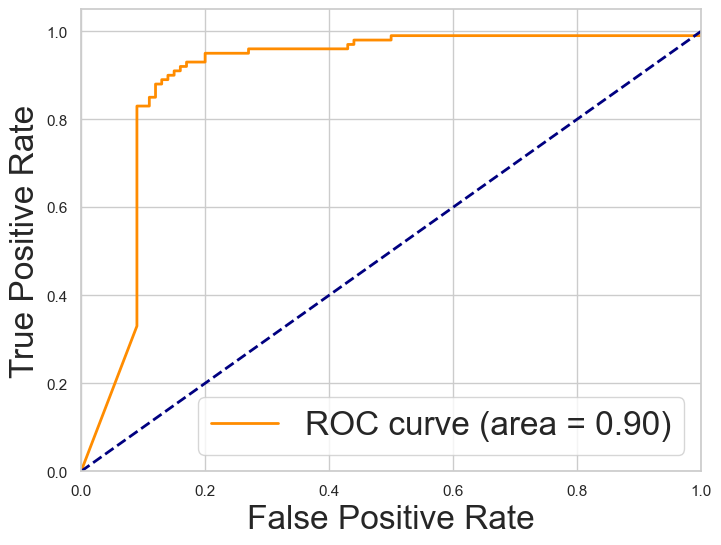}
        \caption{\(\roc\) (Dy. an. I)}
    \end{subfigure}
    \hfill
    \begin{subfigure}[b]{0.24\linewidth}
        \centering
        \includegraphics[width=\linewidth]{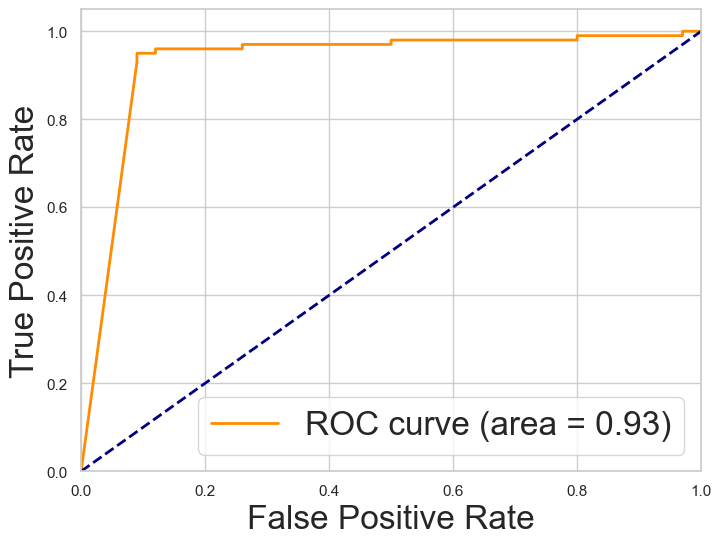}
        \caption{\(\roc\) (Shift)}
    \end{subfigure}
    \caption{Subfigures (a) and (b) show anomalies in the queuing model: blue lines are normal paths
        and red lines are anomalous paths. Subfigures (c) and (d) display the $\roc$ curves of our scoring function for these types of anomalies.}\label{fig:queue_compendium}
\end{figure}


\section{Some Concluding Remarks and Perspectives}\label{sec:perspectives}
In this paper, we have proposed a concept of statistical depth tailored to Markov
data/laws, relying on classic approaches in the multivariate setting. A list of desirable
properties have been shown to be satisfied when the Markov depth is based on popular
choices for the underlying multivariate depth. The concept is promoted throughout the
article as a very natural way of defining an ordering on the set of trajectories of
variable (finite) length, so as to extend the notions of order statistics and ranks for
Markov data. Theoretical and experimental results have also been presented in this
article, showing that Markov depth can be accurately estimated by its empirical version
and may serve to detect anomalous paths in a reliable manner. It may be used to perform a
wide variety of other statistical tasks, ranging from the design of homogeneity tests
between sets of trajectories of possibly different lengths to the clustering of Markov
paths through the design of dedicated visualization techniques. Beyond the experiments we
have presented here for illustrative purposes only, we plan to further investigate the
suitability of this tool for the analysis of Markov data.

\begin{credits}
    \subsubsection{\ackname} This study was funded by a Grant of Hi! Paris foundation.

    \subsubsection{\discintname}
    The authors have no competing interests to declare that are
    relevant to the content of this article.
\end{credits}
%
%
%
\bibliographystyle{splncs04}
\bibliography{biblio}
%





\newpage

\appendix

\section*{Appendix} 

\setcounter{section}{0} 

The Appendix is organized in two parts: Part \ref{sec:technical_appendix} provides
mathematical background, examples, and detailed technical proofs of the results
presented in the main article. Part \ref{sec:additional_experiments} offers expanded
descriptions of the experiments from Section \ref{sec:numerical_experiments} along
with additional numerical examples.

\renewcommand{\thesection}{A}
\section{Technical appendix}\label{sec:technical_appendix}

\setcounter{section}{0} 
\renewcommand{\thesection}{A\arabic{section}}
\renewcommand{\theequation}{\thesection.\arabic{equation}}

\section{Statistical Depths of Multivariate Distributions}\label{subsec:depth_supp}
Here we describe some well known multivariate statistical depth
functions and we state the validity of the properties
    {\PropGeneralInvariance}-{\PropContinuityInProb} described in section
\ref{sec:statistical_depth}\smallskip

\noindent\textbf{Halfspace depth.} The halfspace depth (\(H\Depth\)), introduced by \cite{Tukey1975}, is the first and most
widely known example of a multivariate depth function. For a point \(x\in\Set\) and a
probability distribution \(P\in\Probabilities(\SgB)\), \(\Depth_P(x)\) is defined as
infimum of the measures, according to \(P\), of the halfspaces containing \(x\), that is \(H\Depth_P(x)=\inf_{u\in\mathbb{S}_{d-1}}\{ P(\mathcal{H}_{u,x})\}\),
where $\mathcal{H}_{u,x}=\{y\in \mathbb{R}^d:\; \langle u,\; y-x\rangle \geq 0\}$. The validity of the properties {\PropGeneralInvariance}
to {\PropMonotonicity} for
all \(P\in\Probabilities(\SgB)\) was proved in Theorem 2.1 of \cite{ZuoSerfling00}, while
the property {\PropContinuityDepth} was established in Lemma 6.1 in \cite{DonohoG92}.
Regarding {\PropContinuityInProbUniform}, Theorem A.3 of \cite{NagyGijbels2016} shows
that it holds for all probability measures, \(P\) such that \(P(L)=0\) for all
hyperplanes \(L\subset\Set\). Property {\PropLipschitz} trivially holds for
$\mathcal{A}=\{H:\; H \text{ open half-spaces of } \mathbb{R}^d\}$.
Several algorithms
have been proposed for the exact computation of
\(H\Depth\), with a computational complexity of \(O(n^{d-1}\ln n)\) (see \cite{DyckerhoffM16,NagyDyckerhoffMozharovskyi2020} for an analysis on this subject)
and it has been established that determining the
exact \(H\Depth\) of a single
point in arbitrary dimensions, based on a sample of \(n\) points, is NP-hard \cite{JOHNSON197893}.\smallskip


\noindent\textbf{Lens depth.} Let \(x_1\) and \(x_2\) be two points in \(\Set\) and \(||\circ||\) be a norm in
\(\Set\). The lens \(L(x_1,x_2)\) of \(x_1\) and \(x_2\) is defined as the intersection
of the closed balls with radius \(||x_1-x_2||\), and centered at \(x_1\) and \(x_2\).
Using this concept, \cite{LiuModarres2011} defined the \textit{lens depth}
(\(L\Depth\)) of a point \(x\in\Set\) and a probability \(P\in\Probabilities(\SgB)\) as
the probability that \(x\) belongs to the random lens \(L(X_1,X_2)\), where \(X_1\) and
\(X_2\) are two independent random vectors sampled from \(P\).
\(L\Depth\) is orthogonally invariant when using the Euclidean distance on
and it satisfies \(\PropVanishingBoundary\) (see \cite{LiuModarres2011}). Examples 1, 2 and 3
in \cite{GeenensNietoFrancisci2023} show
that properties {\PropMaximalityCenter} and
    {\PropMonotonicity} are not satisfied in general. Sections 3.3 and 3.4 in \cite{GeenensNietoFrancisci2023} established the validity of
    {\PropContinuityDepth} and {\PropContinuityInProb} under weak conditions. The
computational efficiency of its sampled version is \(O(n^2)\) \cite{MM22}.
\smallskip

\noindent\textbf{Mahalanobis depth.} Given a \(d\)-dimensional distribution \(P\) with
expectation \(\mu_P\) and covariance matrix \(\Sigma_P\), the Mahalanobis norm of \(x\in\R^d\)
w.r.t. \(P\) is defined as \(||{x}||^2_{\Sigma_P}=x^T \Sigma_P^{-1}x\). This norm measures
the distance between the point \(x\) and the distribution \(P\). The Mahalanodis depth
is then defined as: \(MD_P(x)= \left(1+||x-\mu_P||_{\Sigma_P}^2\right)^{-1}\).
It satisfies properties {\PropGeneralInvariance} through {\PropContinuityDepth} for general distributions,
and it satisfies {\PropContinuityInProbUniform} for distributions with regular covariance matrix.
Regarding the computation of its sampling version, the execution times are virtually independent of
(moderate values of) $n$ and $d$ (the actual complexity being $O(n)$). See \cite{MM22} for more details.

\section{The Space $\mathbb{T}$ of Finite Length Paths}\label{subsec:topology_supp}
In order to facilitate the exposure, we will concentrate on the case \(d=1\). The
properties in the case \(d>1\) follow directly from the natural isomorphism between
\((\R^d)^n\) and \(\R^{dn}\). Denote by \(\R^{\infty}\) the
\textit{space of finite sequences over} $\R$, that is
\[
      \R^{\infty} = \left\{ (x_1, x_2, \ldots) \in \R^\mathbb{N} : \text{all but finitely many } x_i \text{ equal } 0 \right\},
\]

For every $n \in \mathbb{N}$, let $\mathrm{In}_{\R^n} : \R^n \to R^{\infty}$ denote the
canonical inclusion \(\mathrm{In}_{\R^n}(x_1, \ldots, x_n) = (x_1, \ldots, x_n, 0, 0,
\ldots)\). Therefore, $\R^{\infty} =
      \cup_{n=1}^{\infty}{\mathrm{In}_{\R^n}(\R^n)}$. This allows us to identify each
\(\mathrm{In}_{\R^n}(\R^n)\) with its corresponding \(\R^n\) and \(\SetSeq\) with
\(\R^{\infty}\). Endow each \(\R^n\) with its natural Borealian topology
\(\borel[\R^n]\). The finite topology \(\SgE\), also called the \textit{weak topology
induced by \(\{\borel[\R^n]\}_{n\in\N}\)} \cite[pp. 131]{Dugundji1978}, is defined as
follows :
\begin{equation*}
      \SgE = \{U\subset\SetSeq: U\cap\R^n\in\borel[\R^n]\;\;\forall n\in\N \}.
\end{equation*}

The following proposition resumes the main topological properties of the space
\((\SetSeq,\SgE)\). These results can be found in sections VI.8 and A.4 in
\cite{Dugundji1978}.

\begin{proposition}\label{propo:properties_finite_topology}{\sc (Properties of \(\SgE\))} The following properties hold true for \(\SgE\).
      \begin{enumerate}[label=(\roman*)]
            \item The space \((\SetSeq,\SgE)\) is complete.
            \item Each \(\R^n\), as a subspace of \((\SetSeq,\SgE)\), retains its
                  original topology \(\borel[\R^n]\).
            \item If \((Y,\mathcal{Y})\) is a topological space, then, a function
                  \(f:\SgE\to Y\) is continuous if and only if, for all \(n\in\N\), the
                  restriction \(f\circ\mathrm{In}_{\R^n}\) is continuous as a function of
                  \((\R^n,\borel[\R^n])\) to \((Y,\mathcal{Y})\).
            \item If \(\{{x_m}\}_{m\in\N}\) converges to \(x=(x_1,\ldots,x_n)\) in
                  \((\SetSeq,\SgE)\), then there exists \(M\) such that \(x_m\in\R^{n}\)
                  for all \(m\geq M\) and \(x_m\to x\) in \(\R^n,\borel[\R^n]\).
            \item \(\SgE\) is smallest topology that makes all the inclusions \(\mathrm{In}_{\R^n}\)
                  continuous (i.e. \((\SetSeq,\SgE)\) is a \textit{coherent space}).
      \end{enumerate}
\end{proposition}

A direct consequence of part (iii) of Proposition \ref{propo:properties_finite_topology}
is that a set \(K\subset \SetSeq\) is compact if and only if, for all \(n\in\N\),
\(K\cap\R^n\) is compact in \(\borel[\R^n]\). We will say that a sequence
\(\{\x_n\}_{n\in\N}\) of elements of \(\SetSeq\) \textit{goes to infinity} when there
exist \(m,M\in\N\) such that \(\x_n\in\R^m\) for all \(n\geq M\) and the sequence
\(\{\x_n\}_{n\geq M}\) goes to infinity in \(\borel[\R^m]\)(\textit{i.e.}
\(\{||\x_n||\}_{n\geq M}\) goes to infinity with \(n\).)

\section{On Markov Kernel Inference}\label{subsec:kernel_supp}
A kernel \(\kernel\) admits a density w.r.t. the measure \(\lambda\)
if there is a non-negative function \(\pi\) defined in \(\Set\times\Set\) such that for
all \(x\in\Set\) and \(A\in\SgB\) we have
$  \kernel(x,A)=\int_{A}{\pi(x,y)d\lambda(y)}$.
For kernels that admit a density, the most widely used estimator of $\pi$ is the
\textit{Nadaraya-Watson} estimator, defined as
\begin{equation} \label{eq:nadaraya}
  \pi_{n}(x,y)=\frac{\sum_{i=0}^{n-1}h_n^{-1}K(h_n^{-1}(x-X_{i}))K(h_n^{-1}(y-X_{i+1}))}
  {\sum_{j=0}^{n-1}K(h_n^{-1}(x-X_{j}))},
\end{equation}
where \(h_n>0\) is a bandwidth parameter and \(K:\Set\to\R\) is a bounded measurable function that
satisfies \(\int_{\Set}|K(x)dx|=1\) and \(||x||\, |K(x)|\to 0\) as \(||x||\to\infty\). Consistency
and asymptotic normality of
this estimator was shown in \cite{AthreyaAtuncar1998} while strong consistency was established
in \cite{dorea2002strong}, and Berry-Essen type inequalities were obtained in \cite{basu1998berry}. In addition to the Nadaraya-Watson estimator, various other estimators have been
suggested in the literature, see for instance, \cite{Clem00,Lacour2007,lacour2008nonparametric}
and the compendium \cite{Sart2014}.\smallskip

\noindent\textbf{Estimating the conditional cdf.} When \(\Set\subset\R\), for \(x\) fixed, integrating over \(y\) in equation
\eqref{eq:nadaraya}, and letting \(G(t):=\int_{-\infty}^{t}{K(t)}\) we get the following
consistent estimator \(\hat{F}_x(y)\) of the cumulative distribution function of
\(X_{1}|X_0=x\)
\begin{equation}\label{eq:nadaray_cdf}
  \hat{F}_x(y)=\frac{\sum_{i=0}^{n-1}{K(h_n^{-1}(x-X_i))G(h_n^{-1}(y-X_{i+1}))}}{\sum_{j=0}^{n-1}K(h_n^{-1}(x-X_{j}))},
\end{equation}
therefore, the distribution function \(\hat{F}_x\) can be written as a weighted sum of
the distribution functions \(G(h_n^{-1}(y-X_{i+1}))\)

\section{On the Symmetry Property}\label{sec:ex_ppties}
We now discuss the property of maximality at path of centers of the
Markovian depth, showing an example of where it holds. We also discuss the concept of
symmetry centers and provide an example that shows why a property like
    {\PropMaximalityCenter} is not useful in the Markovian scenario.\smallskip

\noindent\textbf{Maximality at path of centers.} Consider the case where $d=1$ and the Markovian depth function \(D_{\Pi}(\mathbf{x})\) is
based on Tukey's half-space depth. In this situation, we have: \(\forall (x,y)\in
\mathbb{R}^2\) (\textit{c.f.} \cite[pp. 754]{CuevasFraiman2009})
\begin{equation}\label{eq:tukey_dimension_1}
    D_{\Pi_{x}}(y)=\min\left\{\Pi_x((-\infty,y]),\quad 1-\Pi_x((-\infty,y))+\kernel_x(\{y\}) \right\}.
\end{equation}

Assume that for all $x\in \mathbb{R}$, the distribution $\kernel_x$ is continuous and
non-degenerate. Then for all $x\in\R$, there is a unique $\theta(x)\in \mathbb{R}$ s.t. $\kernel_x((-\infty,\; \theta(x)])=1/2$. The point $\theta(x)$ has maximal Tukey's
depth of $1/2$ with respect to the distribution $\Pi_x$. As a result, the Markov depth
attains its maximum value of $1/2$ at the path $(x_0,\; \theta(x_0),\; \ldots,\;
    \theta^{(n-1)}(x_0))$ for any $n\geq 1$ and \(x_0\in\R\). The following example shows that this situation is quite common.

\begin{example}{\sc (Markov chain with continuous marginals)}
    Suppose that \(\kernel\) admits a non-constant density function $f(x,t)$ w.r.t. the Riemann's integral.
    Then, the distribution functions of \(\kernel_x\) are all continuous and non-degenerated, therefore
    their Markovian depth is maximal at path $(x_0,\; \theta(x_0),\; \ldots,\; \theta^{(n-1)}(x_0))$
    for any $n\geq 1$ and \(x_0\in\R\).
\end{example}

Many well-known Markov models have densities, for example Random Walks, Autoregressive
processes and Random coefficient autoregressive models. For more examples and a detailed
description, see \cite[Chapter 2]{markovChain2018}.\medskip

\noindent\textbf{Symmetry.} A way to define symmetry w.r.t. the null path $\mathbf{0}$
(\textit{i.e.} the infinite path with all components equal to $0$) in the path space is
by requiring that for all \(n\in\N\) and any Borel set $A\subset \mathbb{R}^{n+1}$: $\mathbb{P}\{ (X_0,\; \ldots,\; X_{n})\in A\}=\mathbb{P}\{ (X_0,\; \ldots,\;
    X_n)\in -A \}$, where $-A=\{\x=(x_0,\ldots,x_n) \in \R^{n+1}\text{ s.t. } (-x_0,\;
    \ldots,\; -x_n)\in A\}$. The following example exhibits a chain that fulfills this symmetry definition, but for which the null trajectory is unrealizable.

\begin{example}{\sc (Unrealizable symmetric path)} Consider the Random Walk defined as \(X_0=0\),
    \(X_{n+1}=X_n+\epsilon_n\) for \(n\geq 0\) with \(\{\epsilon_n\}_{n\geq 0}\) being an
    i.i.d. sequence of random variables such that
    \(\P(\epsilon_n=1)=\P(\epsilon_n=-1)=1/2\). This model is symmetric w.r.t the
    infinite null trajectory $\mathbf{0}$, however, this trajectory is unrealizable, as \(\P(X_{n+1}=X_{n})=0\) for all \(n\in \mathbb{N}\).
\end{example}


\section{Examples of Markov Depths}\label{sec:examples}
In this section we show the form of the Markovian depth in some scenarios.

\begin{example}{\sc (Markov Tukey's depth)}\label{ex:markov_tukey_d_1}
    Suppose \(\chain\) is a Markov chain with kernel \(\kernel\) taking values in \(E\subseteq\R\).
    Let \(\x=(x_0,\ldots,x_n)\in\SetSeq\) be such that \(\kernel_{x_i}\) is a continuous
    random variable for \(i=0,\ldots,n-1\). Then, the Markov depth \(\DepthSeq^H\), based on
    the halfspace depth has the form

    \begin{equation}
        \Depth^{H}_{\Pi}(\mathbf{x})=\sqrt[n]{\prod_{i=1}^{n}{\min{\left\{ \kernel_{x_{i-1}}\left((-\infty,x_{i}]\right), 1-\kernel_{x_{i-1}}\left((-\infty,x_{i}]\right)  \right\}}}}.
    \end{equation}
\end{example}




We now give
the form of the sampling versions of the Markovian depth for real valued Markov chains
with densities. We use the notation from section \ref{subsec:kernel_supp}.

\begin{example}{\sc (Sampling versions in \(\R\))}\label{ex:empirical_depth} Let \(\chain\) be a real valued,
    positive recurrent Markov chain with kernel
    \(\kernel\), having continuous density \(\pi\). Suppose we have a realization
    \(X=(X_0,\ldots,X_n)\) of \(\chain\), and let \(\x=(x_0,\ldots,x_m)\in\SetSeq\).
    Under this hypothesis, the Nadaraya-Watson estimator \eqref{eq:nadaray_cdf}, based on
    \(X\) can be used, and the estimators \(\hat{F}_{x_{i-1}}(y)\) of the marginal CDFs
    are continuous functions of \(y\). By example \ref{ex:markov_tukey_d_1},
    the sampling version of the
    Markov depth \(\DepthSeq^H\) adopts the form:
    \begin{equation*}
        \Depth^{H}_{\hat{\Pi}}(\mathbf{x}) =\sqrt[m]{\prod_{i=1}^{m}{ \min\left\{ \hat{F}_{x_{i-1}}(x_{i}), 1-\hat{F}_{x_{i-1}}(x_{i}) \right\} }},
    \end{equation*}
    where the $\hat{F}_{x_{i-1}}(x_{i})$ are defined as in
    \eqref{eq:nadaray_cdf}. Notice that the lengths of \(X\) and \(\x\) do not
    necessarily coincide.
\end{example}


\section{Technical Proofs}\label{sec:proofs}
The proofs of the theoretical results stated in the paper are detailed below.\medskip

\noindent\textbf{Proof of Proposition \ref{prop:limit}:}\label{subsec:limit_depth_supp}
As we $\mathbb{P}_{\nu}$ almost-surely have: $\forall n\geq 1$, $$ D_{\Pi}(X_0,\;
    \ldots,\;
    X_n)=\exp\left(\frac{1}{n}\sum\limits_{i=1}^{n}\log(D_{\Pi_{X_{i-1}}}(X_i))\right), $$
the result straightforwardly follows from the SLLN and the CLT for additive functionals
of aperiodic Harris recurrent Markov chains (refer to \textit{e.g.} Theorem 4.1 in
\cite{Tjostheim-2001}), combined with the delta method applied to the differentiable
function $x\mapsto \exp(x)$ (see \textit{e.g.} Theorem 3.1 in \cite{vdV98}). The
asymptotic variance is given by
\begin{equation}\label{eq:asymptotic_vaiance}
    \sigma^2_{\Pi}=\int_{(x,y)}
    \log^2(D_{\Pi_x}(y))\invariantMeasure d(x)\Pi_x(dy)-\int_{x}{\left(\int_{y}\log(D_{\Pi_x}(y))\Pi(x, dy)\right)^2}\invariantMeasure(dx).
\end{equation}


\noindent\textbf{Proof of Theorem \ref{thm:ZS_properties}:}
\textit{Independence from the initial law:} This property follows directly from definition \ref{def:markovian_depth} as the Markovian
sample path depth is defined in terms of the transition probability kernel only.

\vspace{-10pt}\paragraph{Affine invariance:}
Let \(f(x)=Ax+b\) be an affine invariant transformation. From the definition
\ref{def:markovian_depth}, it is clear that if we show
\(\Depth_{f(\kernel)_{f(y_0)}}\left(f(y_1)\right)=\Depth_{\kernel_{y_0}}(y_1)\) for all,
\(y_0, y_1\in\Set\) then the claim of the theorem follows immediately.

Denote by \(F_{y_0}\) the distribution of \(f(Y)\), where \(Y\) is a random variable with
distribution \(\kernel_{y_0}\). The affine invariance of \(\Depth\) implies that
\begin{equation}\label{eq:invariance_underlying_dist}
    \Depth_{F_{y_0}}\left(f(y_1)\right)=\Depth_{\kernel_{y_0}}(y_1).
\end{equation}

On the other hand, notice that, for any \(C\in\SgB\) and \(y\in\Set\) we have
\begin{equation*}
    f(\kernel)(y,C)=\P\bigl(f\left(X_1\right)\in C| f\left(X_0\right)=y\bigr)=\kernel\left(f^{-1}(y), f^{-1}(C)\right),
\end{equation*}
therefore,
\(f(\kernel)_{f(y_0)}(C)=\kernel(y_0,f^{-1}(C))=\kernel_{y_0}(f^{-1}(C))=F_{y_0}(C)\),
which shows that the distributions \(F_{y_0}\) and
\(f(\kernel)_{f(y_0)}\) coincide. The result now follows from
\eqref{eq:invariance_underlying_dist}.

\vspace{-10pt}\paragraph{Vanishing at infinity:} We need the following lemma.

\begin{lemma}\label{lemma:d_m_lower_bound}
    Let \(\x=(x_0,\ldots,x_{m})\in\Set^{m+1}\). For any \(\epsilon>0\), if there is \(j\)
    such that \(D(x_j, \kernel_{x_{j-1}} )<\epsilon\), then
    \(\DepthSeq(\x)<\epsilon^{1/m}\).
\end{lemma}
\begin{proof}[Lemma \ref{lemma:d_m_lower_bound}]
    Using that \(0<\Depth<1\), we have
    \(\sqrt[m]{\prod_{i=1}^{m}D_{\Pi_{x_{i-1}}}(x_i)}\leq \epsilon^{1/m}\).
\end{proof}

Let \(\{\x_n\}_{n=1}^{\infty}\subset\SetSeq\) be a sequence that tends to infinity in
\(\SgE\). The topology of the torus implies that there exists an integer \(m\) such
\(\x_n\in\Set^{m+1}\) for \(n\) big enough and that \(||\x_n||_{\infty}:=\max_{0\leq i
    \leq m}||x_{i,n}||\) tends to infinity with \(n\). Suppose that \(\DepthSeq(x_n)\) does
not converge to \(0\), then, there is an \(\epsilon>0\) such that
\(\DepthSeq(x_n)>\epsilon\) for all \(n\). Therefore, by the contrapositive of Lemma
\ref{lemma:d_m_lower_bound}, for all \(n\in\N\) and \(i=0,\ldots,m\) we have
\(\Depth_{\kernel_{x_{i-1,n}}}(x_{i,n})>\epsilon^{m}\).

Because \(\Depth_{\kernel_{x_0}}\) satisfies {\PropVanishingBoundary}, we have that there
exists \(C_1\) such that \(\Depth_{\kernel_{x_0}}(y)<\epsilon^m\) for all \(y\) with
\(||y||>C_1\), therefore, \(||x_{1,n}||\leq C_1\) for all \(n\). By the same reasoning,
but using that \(\sup_{||x||\leq C_1}D_{\Pi_x}(y)\to 0\) as \(\vert\vert y\vert\vert \to
\infty\), we obtain that there exists \(C_2\) such that \(||x_{2,n}||\leq C_2\) for all
\(n\). Repeating the same argument for \(i=2,\ldots,m\), we get that the sequences
\(\{x_{2,n}\}_{n\geq 1},\ldots,\{x_{m,n}\}_{n\geq 1}\) are all bounded, therefore
\(||\x_n||_{\infty}\) is also bounded, which contradicts our assumption that
\(||\x_n||_{\infty}\) tends to infinity.

\noindent\textbf{Proof of Proposition \ref{thm:MaximalityCenter}:}
This is a direct consequence of the monotonicity of the exponential and logarithmic
functions.

\noindent\textbf{Proof of Theorem \ref{th:continuity_in_x}:}\label{proof:continuity_in_x}
\textit{Continuity in \(\x\):} If \(\x_n\) is a sequence in \(\SetSeq\) that converges to \(\x\), then
\(\x_n\in\Set^{m+1}\) for \(n\) big enough. Moreover, \(\x_n=(x_{0,n},\ldots,x_{m,n})\)
converges to \(\x=(x_{0},\ldots,x_{m})\) in \(\Set^{m+1}\) if and only if \(x_{i,n}\to
x_i\) in \(\Set\) for \(i=1,\ldots,m\). Notice that for each \(i=1,\ldots,m\) and \(n\)
big enough we have
\begin{multline*}
    |\Depth_{\kernel_{x_{i-1,n}}}(x_{i,n}) - \Depth_{\kernel_{x_{i-1}}}(x_{i})|\leq|\Depth_{\kernel_{x_{i-1,n}}}(x_{i,n}) - \Depth_{\kernel_{x_{i-1}}}(x_{i,n}) |+\\
    |\Depth_{\kernel_{x_{i-1}}}(x_{i,n}) - \Depth_{\kernel_{x_{i-1}}}(x_{i})|.
\end{multline*}

The first term on the right hand side of the previous inequality tends to \(0\) as \(n\)
tends to infinity thanks to property {\PropContinuityInProbUniform} of \(\Depth\) and the
weakly convergence of \(\kernel_{x_{i-1,n}}\) to \(\kernel_{x_{i-1}}\). Similarly, the
second term goes to \(0\) as \(n\) goes to infinity in virtue of {\PropContinuityDepth}.
This shows that \(\Depth_{\kernel_{x_{i-1,n}}}(x_{i,n})\to_n
\Depth_{\kernel_{x_{i-1}}}(x_{i})\) for \(i=1,\ldots,m\). The continuity of \(\DepthSeq\)
now follows from \eqref{def:markovian_depth} and the continuity of \(y\mapsto\sqrt[m]{y}\).

\paragraph{Continuity in \(\kernel\):} For each \(\x\in\SetSeq\) we have that \(|\DepthSeq[\kernel^{(n)}](\x) - \DepthSeq[\kernel](\x)|\)
is smaller than \(\max_{0\leq i \leq m-1}{|\Depth_{\kernel^{(n)}_{x_i}}(x_{i+1}) - \Depth_{\kernel_{x_i}}(x_i+1)|}\),
the result now follows from {\PropContinuityInProb}.

\noindent\textbf{Proof of Theorem \ref{th:non_asymptotic_bound}:}
We prove the following stronger version of Theorem \ref{th:non_asymptotic_bound}

\begin{theorem}\label{th:strong_non_asymptotic_bound}
    Let $n\geq 1$ and $\mathbf{x}=(x_0,\; \ldots,\; x_n)\in E^{n+1}$. Consider two
    transition probabilities $\Pi$ and $\hat{\Pi}$ on $E$. Suppose that \(\Depth\) fulfills
    property {\PropLipschitz} and $C_d||\kernel_{x_i} - \hat{\kernel}_{x_i}||_{\mathcal{A}}<\min\{\Depth_{\kernel_{x_i}}(x_{i+1}), \Depth_{\hat{\kernel}_{x_i}}(x_{i+1})\}$ for $i=1,\; \ldots,\; n$. Then,
    \begin{equation}\label{eq:strong_bound}
        |\DepthSeq(\x) - \DepthSeq[\hat{\kernel}](\x)| \leq \frac{7}{4}C_d\frac{\DepthSeq(\x)}{H_{\kernel, \hat{\kernel}}(\x)} \max_{i=0,\; \ldots,\; n-1}||\kernel_{x_i} - \hat{\kernel}_{x_i}||_{\mathcal{A}},
    \end{equation}
    where
    \begin{equation*}
        H_{\kernel, \hat{\kernel}}(\mathbf{x}):=\frac{n}{\sum_{i=0}^{n-1} { 1/\sqrt{ \Depth_{\kernel_{x_i}}(x_{i+1})\Depth_{\hat{\kernel}_{x_i}}(x_{i+1}) }} }
    \end{equation*}
    represents the harmonic mean of  $\sqrt{ \Depth_{\kernel_{x_i}}(x_{i+1})\Depth_{\hat{\kernel}_{x_i}}(x_{i+1}) },\; i=0,\ldots,n-1 $.
\end{theorem}
\begin{proof}[Theorem \ref{th:strong_non_asymptotic_bound}]
    Let \(h_i=\DepthSeq[\hat{\kernel}_{x_i}](x_{i+1})-\DepthSeq[\kernel_{x_i}](x_{i+1})\), then, we can write \(\ln{\DepthSeq[\hat{\kernel}_{x_i}](x_{i+1})} = \ln{\DepthSeq[\kernel_{x_i}](x_{i+1})}+\ln(1+{h_i}/{\DepthSeq[\kernel_{x_i}](x_{i+1})})\)
    and
    \begin{equation}
        |\DepthSeq(\x)-\DepthSeq[\hat{\kernel}](\x)|=\DepthSeq(\x)\left| 1-\exp{\left(\frac{1}{n}\sum_{i=0}^{n-1}\ln\left(1+\frac{h_i}{\DepthSeq[\kernel_{x_i}](x_{i+1})}\right)\right)}  \right|.\label{eq:bound_diff_depths}
    \end{equation}
    By {\PropLipschitz} and our hypothesis, we have \[|h_i|\leq C_d||\kernel_{x_i} -
        \hat{\kernel}_{x_i}||_{\mathcal{A}}<\min\{\Depth_{\kernel_{x_i}}(x_{i+1+}),
        \Depth_{\hat{\kernel}_{x_i}}(x_{i+1})\}\]  for \(i=0,\ldots,n-1\),
    therefore, \(|h_i|/
    \min\{\Depth_{\kernel_{x_i}}(x_{i+1}), \Depth_{\hat{\kernel}_{x_i}}(x_{i+1})\}<1\)
    for all \(i\). Using the inequality \(|\ln{(1+t)}|\leq |t|/\sqrt{(1+t)}\), valid for
    \(t\geq -1\) \cite[pp. 160]{Bullen1998}, and that
    \((h_i/\Depth_{\kernel_{x_i}}(x_{i+1}))(1+h_i/\Depth_{\kernel_{x_i}}(x_{i+1}))^{-1}=h_i/\Depth_{\hat{\kernel}_{x_i}}(x_{i+1})\),
    we get
    \begin{align}
        \left| \ln\left(1+\frac{h_i}{\DepthSeq[\kernel_{x_i}](x_{i+1})}\right) \right| & \leq \frac{|h_i|}{\sqrt{ \Depth_{\kernel_{x_i}}(x_{i+1})\Depth_{\hat{\kernel}_{x_i}}(x_{i+1}) }}, \; i=0,\ldots,n-1\label{eq:log_bound}
    \end{align}
    which is smaller than \(1\). Therefore, \(|\frac{1}{n}\sum_{i=0}^{n-1}\ln
    (1+{h_i}/{\DepthSeq[\kernel_{x_i}](x_{i+1})})|<1\). This allows us to apply the
    inequality \(|1-e^t|<\frac{7}{4}|t|\) (valid for \(|t|<1\))\cite[pp. 82]{Bullen1998}
    on equation \eqref{eq:bound_diff_depths}, which combined with equation
    \eqref{eq:log_bound} shows that
    \begin{align*}
        |\DepthSeq(\x)-\DepthSeq[\hat{\kernel}](\x)| & \leq \frac{7}{4}\DepthSeq(\x)\frac{1}{n} \sum_{i=0}^{n-1}{\frac{|h_i|}{\sqrt{ \Depth_{\kernel_{x_i}}(x_{i+1})\Depth_{\hat{\kernel}_{x_i}}(x_{i+1}) }}}.
    \end{align*}
    Equation \eqref{eq:strong_bound} now follows from {\PropLipschitz}.
\end{proof}

Now we proceed to the proof of Theorem \ref{th:non_asymptotic_bound}. Notice that
if \(C_d||\kernel_{x_i} - \hat{\kernel}_{x_i}||_{\mathcal{A}}\) is greater than \(
\min\{\Depth_{\kernel_{x_i}}(x_{i+1}), \Depth_{\hat{\kernel}_{x_i}}(x_{i+1})\}\) for some
\(i\in\{0,\; \ldots,\; n-1\}\), then \( C_d \max_{i=0,\ldots,n-1}||\kernel_{x_i} -
\hat{\kernel}_{x_i}||_{\mathcal{A}}>\epsilon\), which implies that the right hand side of
\eqref{eq:bound} is bigger than \((7/4)\DepthSeq(\x)\). In this case, the
inequality is trivially satisfied because \(\DepthSeq(\x)\) and
\(\DepthSeq[\hat{\kernel}](\x)\) are both between \(0\) and \(1\). On the other hand,
when \(C_d||\kernel_{x_i} - \hat{\kernel}_{x_i}||_{\mathcal{A}}\) is smaller than \(
\min\{\Depth_{\kernel_{x_i}}(x_{i+1}), \Depth_{\hat{\kernel}_{x_i}}(x_{i+1})\}\) for all
\(i\), we can apply Theorem \ref{th:strong_non_asymptotic_bound}. Equation
\eqref{eq:bound} now follows from \eqref{eq:strong_bound} using that
\(\sqrt{ \Depth_{\kernel_{x_i}}(x_{i+1})\Depth_{\hat{\kernel}_{x_i}}(x_{i+1}) }\) is greater than \(\min
\{D_{\hat{\kernel}_{i}}(x_{i+1}),D_{\kernel_{i}}(x_{i+1})\}>\epsilon\).

\setcounter{section}{0} 
\renewcommand{\thesection}{B}
\section{Additional Numerical Experiments}\label{sec:additional_experiments}

\setcounter{section}{0} 
\renewcommand{\thesection}{B\arabic{section}}
\renewcommand{\theequation}{\thesection.\arabic{equation}}

This appendix provides additional numerical experiments demonstrating the practical
utility of the Markov depth concept. Section \ref{sec:supp_anomaly_detection} continues
the experiments from Section~\ref{sec:numerical_experiments} by comparing our method with
competitors. Section \ref{sec:visualization_clustering} presents
Markovian DD-plots for clustering, and Section \ref{sec:homogeneity_testing} shows
how to use Markov depth for homogeneity testing. The code for all experiments, as well
as high resolution versions of the images presented, is
available at {\githubRepo}.

\section{Anomaly Detection: a Benchmark Study}\label{sec:supp_anomaly_detection}\label{sec:comparison_competitors}

In order to compare our proposal (\(\DepthSeq\)) against the most commonly used methods
for anomaly detection in multivariate data, specifically Isolation Forest (IF), Local
Outlier Factor (LOF), and Mahalanobis Depth (MD), we must limit our analysis to scenarios
where all trajectories are of equal length. For this experiment, we generated 8 datasets
(one for each type of anomaly and model), each containing 100 trajectories of 200 points,
with 5\% of the trajectories exhibiting a specific type of anomaly. We then applied our
method alongside the previously mentioned anomaly detection algorithms to each dataset.
For IF and LOF we have used the implementation contained in the python library
\textit{sklearn}, while for the Mahalanobis depth we have used the
Python package \textit{data-depth}, all parameters used in the algorithms were set to
their default values. For our Markovian depth algorithm, we generated a sample of 10
trajectories of size 200 in order to obtain an estimator of the kernel. The \(\roc\)
curves (Fig. \ref{fig:arch_roc_comparison} and \ref{fig:queue_roc_comparison}) and the
AUC (Table \ref{tab:both_auc_comparison}) clearly show that, for classical anomalies
(shock and shift) our method performs similarly to the best between IF, LOF and MD, while
for dynamic anomalies, it outperforms it.

\vspace{-25pt}\begin{table}[h]
    \caption{Comparison of the AUC for different classifiers.}
    \vspace{5pt}
    \centering
    \label{tab:both_auc_comparison}
    \begin{tabular}{lccccccccc}
        \toprule
                           & \multicolumn{4}{c}{ARCH(1) model} &      & \multicolumn{4}{c}{Queuing model}                                                         \\
        \cmidrule(lr){2-5} \cmidrule(lr){7-10}
        Anomaly type       & IF                                & LOF  & MD                                & \(\DepthSeq\) &  & IF   & LOF  & MD   & \(\DepthSeq\) \\
        \midrule
        Shock              & 0.75                              & 0.82 & 0.68                              & 0.85          &  & 0.67 & 0.98 & 0.44 & 0.98          \\
        Dynamic anomaly I  & 0.42                              & 0.63 & 0.52                              & 0.84          &  & 0.59 & 0.80 & 0.80 & 0.94          \\
        Dynamic anomaly II & 0.68                              & 0.90 & 0.91                              & 0.99          &  & 0.64 & 0.55 & 0.50 & 0.85          \\
        Shift              & 1                                 & 1    & 0.6                               & 1             &  & 0.93 & 0.99 & 0.69 & 0.98          \\
        \bottomrule
    \end{tabular}
\end{table}

\vspace{-25pt}\begin{figure}[h]
    \centering
    \begin{subfigure}[b]{0.49\textwidth}
        \centering
        \includegraphics[height=3cm, width=\textwidth]{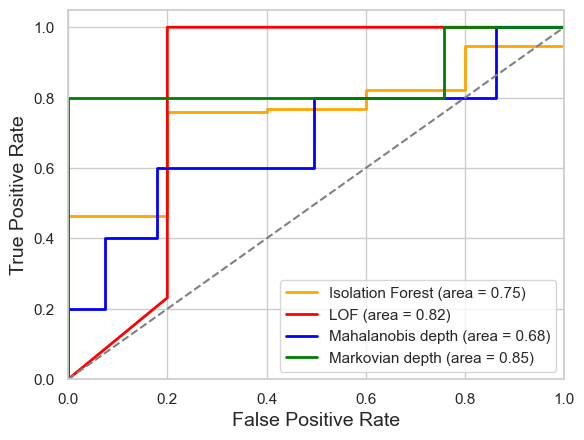}
        \caption{Shock}
    \end{subfigure}
    \hfill
    \begin{subfigure}[b]{0.49\textwidth}
        \centering
        \includegraphics[height=3cm, width=\textwidth]{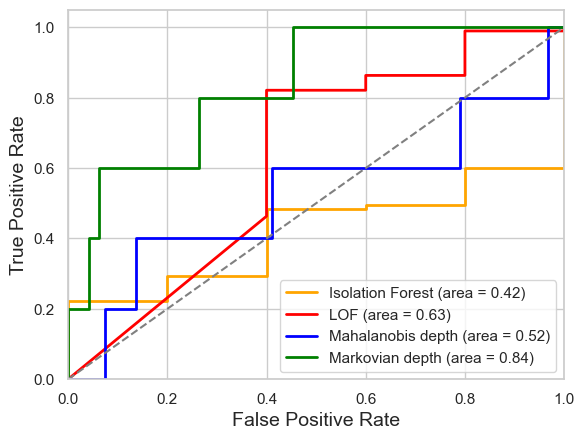}
        \caption{Perturbed mean (Dyn. an. II)}
    \end{subfigure}

    \begin{subfigure}[b]{0.49\textwidth}
        \centering
        \includegraphics[height=3cm, width=\textwidth]{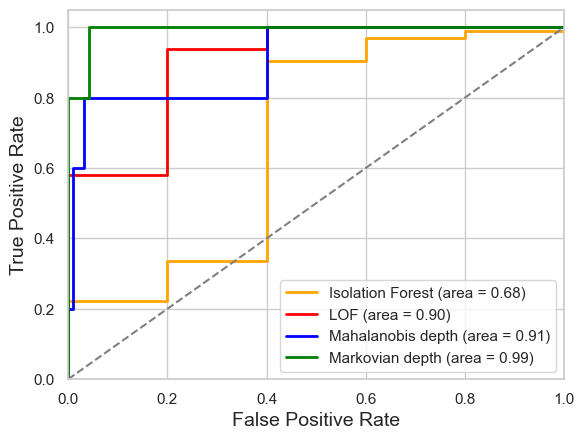}
        \caption{Increasing volatility (Dyn. an. II)}
    \end{subfigure}
    \hfill
    \begin{subfigure}[b]{0.49\textwidth}
        \centering
        \includegraphics[height=3cm, width=\textwidth]{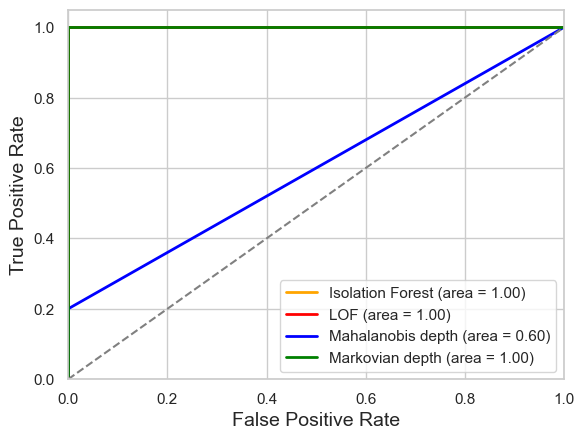}
        \caption{Deterministic mean (Shift)}
    \end{subfigure}

    \caption{\(\roc\) curves for each type of anomaly in the ARCH(1) model.}\label{fig:arch_roc_comparison}
\end{figure}

\begin{figure}[h]
    \centering
    \begin{subfigure}[b]{0.49\textwidth}
        \centering
        \includegraphics[height=2.8cm, width=\textwidth]{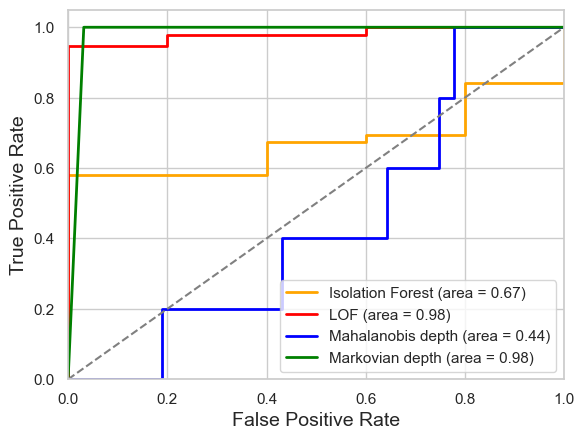}
        \caption{Shock}
    \end{subfigure}
    \hfill
    \begin{subfigure}[b]{0.49\textwidth}
        \centering
        \includegraphics[height=2.8cm, width=\textwidth]{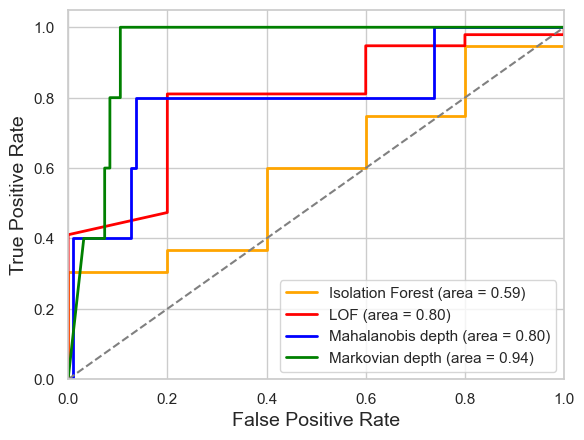}
        \caption{Faster arrivals (Dyn. an. I)}
    \end{subfigure}

    \begin{subfigure}[b]{0.49\textwidth}
        \centering
        \includegraphics[height=2.8cm, width=\textwidth]{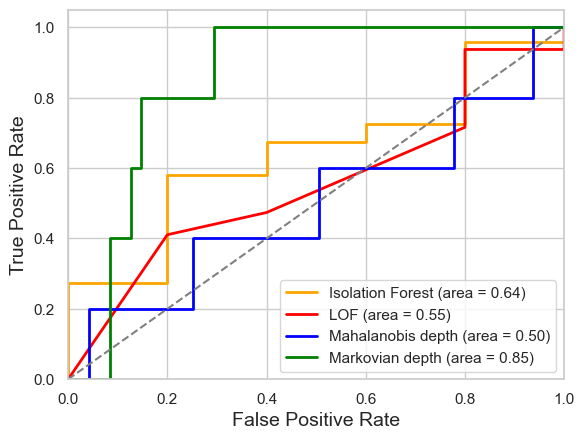}
        \caption{Slower services (Dyn. an. II)}
    \end{subfigure}
    \hfill
    \begin{subfigure}[b]{0.49\textwidth}
        \centering
        \includegraphics[height=2.8cm, width=\textwidth]{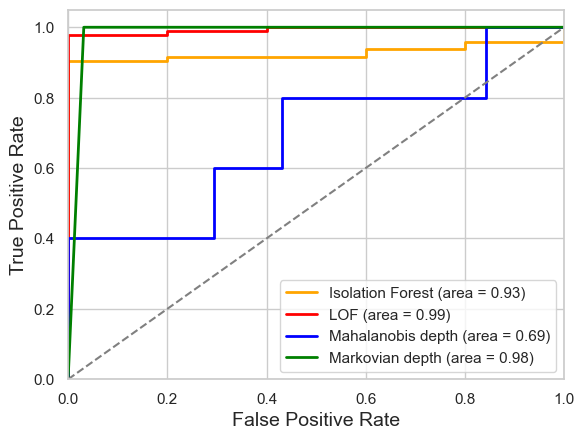}
        \caption{Deterministic arrivals (Shift)}
    \end{subfigure}

    \caption{\(\roc\) curves for each type of anomaly in the queuing model.}\label{fig:queue_roc_comparison}
\end{figure}

\vspace{-12pt}\section{Clustering Markov Paths}\label{sec:visualization_clustering}
\vspace{-5pt}The infinite dimension of Markov chains makes it challenging to visualize the data and
understand the underlying structure. As it was shown in Figure \ref{fig:trajectories} even in the case \(d=1\) and
constant path length, plotting the paths as a function of \(n\) is not very useful, as
the trajectories tend to occupy a vast part of the space due to the ergodicity. Markovian depth functions can be effectively utilized to develop a visual diagnostic tool
for Markovian paths, building upon the Depth vs. Depth plot (DD-plot), first introduced
in \cite{LiuPareliusSingh1999} for multivariate data. For two sets of Markov
trajectories\ \(\mathcal{X}=\{\x_1,\ldots,\x_{n_1}\}\) and
\(\mathcal{Y}=\{\y_1,\ldots,\y_{n_2}\}\) with corresponding kernel estimators
\(\hat{\Phi}\) and \(\hat{\Psi}\), the Markovian DD-plot is obtained by plotting in the
Euclidean plane the points \(\left\{\left(\DepthSeq[\hat{\Phi}](\x),\DepthSeq[\hat{\Psi}](\x)\right): \x\in\mathcal{X}\cup\mathcal{Y} \right\}.\)

To illustrate the use of DD-plots, we have generated 5 data sets, each one containing 50
trajectories of random lengths (between 50 and 200 steps). These datasets,
labeled \(\mathcal{X},\mathcal{Y}_a, \mathcal{Y}_b, \mathcal{Y}_c, \mathcal{Y}_d \), are
constructed following an ARCH(1) model with the parameters stated in Table
\ref{tab:parameters_arch_1}. Figure \ref{fig:arch_dd_plots}, shows the DD-plots of
\(\mathcal{X}\) (blue points) and \(\mathcal{Y}_i\) (yellow stars) for \(i=a,b,c,d\).
It can be readily seen that, the more ``similar'' the distributions of both sets are (in the
sense of the parameters used to generate the data), the more commingled the points in
the DD-plot are. Notice that the samples in \(\mathcal{Y}_a\) have the same distribution
as the samples in \(\mathcal{X}\), then, it is expected that the DD-plot shows a
straight line, see \ref{fig:arch_same_dist}. Conversely, the more the distributions
diverge, the more separated the points (see Fig. \ref{fig:arch_constant_vol}).

\begin{table}[ht]
    \caption{Parameters of the ARCH(1) model used in Figure \ref{fig:arch_dd_plots}}
    \vspace{5pt}
    \centering
    \label{tab:parameters_arch_1}
    \begin{tabular}{lcc}
        \toprule
        Dataset           & $m(x)$                  & $\sigma(x)$                    \\
        \midrule
        \(\mathcal{X}\)   & \({(1+\exp(-x))}^{-1}\) & \(\psi(x+1.2)+1.5\psi(x-1.2)\) \\
        \(\mathcal{Y}_a\) & \({(1+\exp(-x))}^{-1}\) & \(\psi(x+1.2)+1.5\psi(x-1.2)\) \\
        \(\mathcal{Y}_b\) & \((2+\exp(-x))^{-1}\)   & \(\psi(x+1.2)+1.5\psi(x-1.2)\) \\
        \(\mathcal{Y}_c\) & \((4+\exp(-x))^{-1}\)   & \(\psi(x+1.2)+1.5\psi(x-1.2)\) \\
        \(\mathcal{Y}_d\) & \({(1+\exp(-x))}^{-1}\) & 1                              \\
        \bottomrule
    \end{tabular}
\end{table}

\begin{figure}[ht]
    \centering
    \begin{subfigure}[b]{0.24\textwidth}
        \centering
        \includegraphics[height=3cm, width=\textwidth]{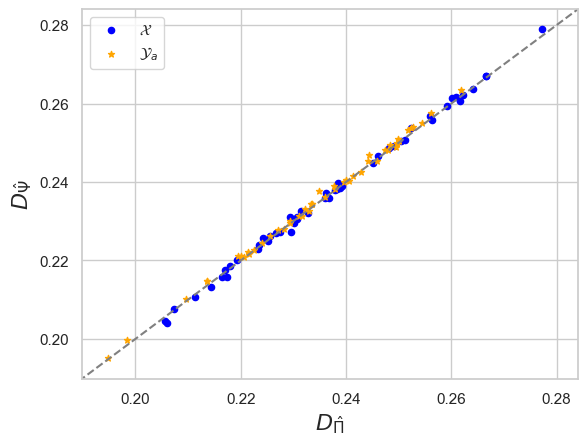}
        \caption{}\label{fig:arch_same_dist}
    \end{subfigure}
    \begin{subfigure}[b]{0.24\textwidth}
        \centering
        \includegraphics[height=3cm, width=\textwidth]{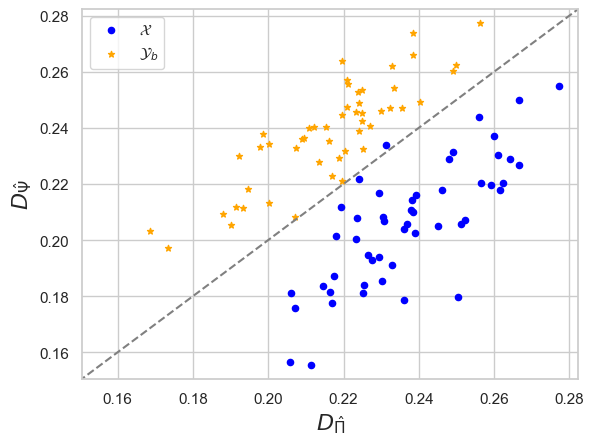}
        \caption{}
    \end{subfigure}
    \begin{subfigure}[b]{0.24\textwidth}
        \centering
        \includegraphics[height=3cm, width=\textwidth]{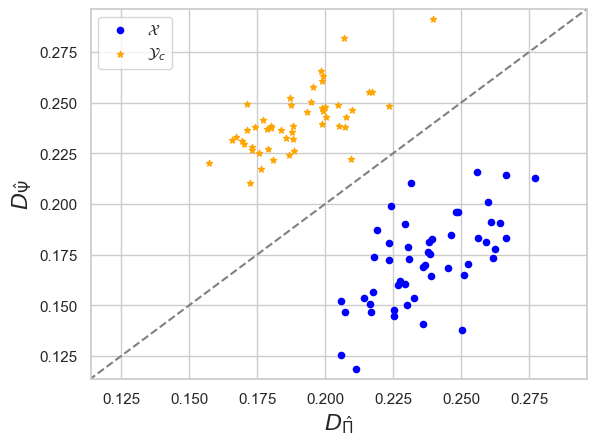}
        \caption{}
    \end{subfigure}
    \begin{subfigure}[b]{0.24\textwidth}
        \centering
        \includegraphics[height=3cm, width=\textwidth]{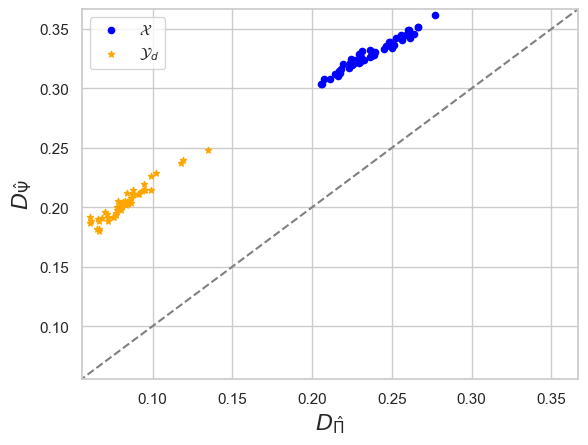}
        \caption{}\label{fig:arch_constant_vol}
    \end{subfigure}
    \caption{Markovian DD-plots for ARCH(1) with parameters described in Table \ref{tab:parameters_arch_1}.}\label{fig:arch_dd_plots}
\end{figure}


\vspace{-3pt}\section{Homogeneity Testing}\label{sec:homogeneity_testing}
The proposed Markovian depth can further be used to provide a formal inference, which we
exemplify as a nonparametric test of homogeneity between the distribution of two Markov
chains. To do this, we will use the depth-based Wilcoxon rank-sum test, proposed in
\cite{LiuSingh1993}. Consider two ARCH(1) models, one with \(m(x)={(1+\exp(-x))}^{-1}\) and
\(\sigma(x)=\psi(x+1.2)+1.5\psi(x-1.2)\), and the other with
\(m(x)={(1+\alpha+\exp(-x))}^{-1}\) and the same \(\sigma\) as the first one. For
different values of \(\alpha\), ranging from 0.3 (very similar distributions) to 1 (reasonably
different) we perform the aforementioned depth-based Wilcoxon rank-sum test, generating
50 trajectories of random length (between 50 and 200) for each chain and using a
reference trajectory of length 1000 from the first chain. Figure
\ref{fig:homogeneity_test} presents the \(p\)-values of the test for each \(\alpha\),
averaged over 50 repetitions. The \(p\)-values effectively identify the differences
between the two distributions when they exist.

\begin{figure}[ht]
    \centering
    \includegraphics[height=3cm, width=0.55\textwidth]{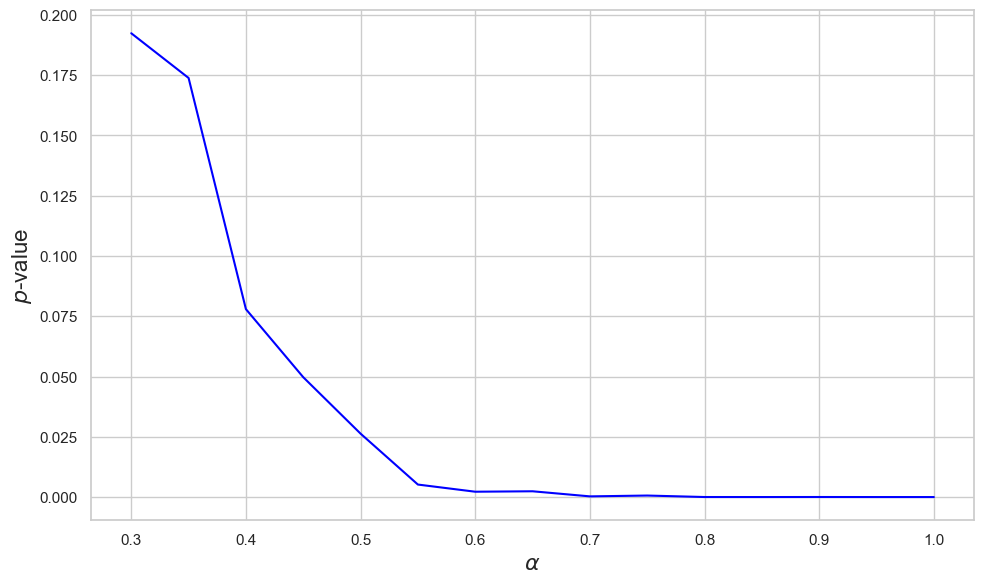}
    \caption{Average \(p\)-values (as a function of \(\alpha\)) for the homogeneity test. }\label{fig:homogeneity_test}
\end{figure}





\end{document}